\documentclass{amsart}[10pts]
\usepackage{amssymb,amsfonts, color,epsf}
\usepackage [cmtip,arrow]{xy}
\xyoption{all}
\usepackage {pb-diagram,pb-xy}
\usepackage{enumerate}
\usepackage{accents}
\usepackage[noautoscale]{youngtab}
\usepackage{subfigure}

\ifx\pdfpageheight\undefined
   \usepackage[dvips,colorlinks=true,linkcolor=blue,citecolor=red,%
      urlcolor=green]{hyperref}
   \usepackage[dvips]{graphicx}
   \makeatletter
   \edef\Gin@extensions{\Gin@extensions,.mps}
   \makeatother
\else
 \usepackage[pdftex]{graphicx}
   \usepackage[bookmarksopen=false,pdftex=true,breaklinks=true,%
      backref=page,pagebackref=true,plainpages=false,%
      hyperindex=true,pdfstartview=FitH,colorlinks=true,%
      pdfpagelabels=true,colorlinks=true,linkcolor=blue,%
      citecolor=red,urlcolor=green,hypertexnames=false%
      ]%
   {hyperref}
\fi
\usepackage{bm}

\usepackage{tikz-cd}
\makeatletter
\tikzset{
  column sep/.code=\def\pgfmatrixcolumnsep{\pgf@matrix@xscale*(#1)},
  row sep/.code   =\def\pgfmatrixrowsep{\pgf@matrix@yscale*(#1)},
  matrix xscale/.code=%
    \pgfmathsetmacro\pgf@matrix@xscale{\pgf@matrix@xscale*(#1)},
  matrix yscale/.code=%
    \pgfmathsetmacro\pgf@matrix@yscale{\pgf@matrix@yscale*(#1)},
  matrix scale/.style={/tikz/matrix xscale={#1},/tikz/matrix yscale={#1}}}
\def\pgf@matrix@xscale{1}
\def\pgf@matrix@yscale{1}
\makeatother

\usepackage{amsmath,amssymb,amsfonts}
\newtheorem{theorem}{Theorem}

\newtheorem{corollary}{Corollary}
\newtheorem{proposition}{Proposition}[section]
\newtheorem{claim}{Claim}[section]
\newtheorem*{claim*}{Claim}

\newtheorem{property}{Property}[section]

\newtheorem*{theorem*}{Theorem}
\newtheorem*{corollary*}{Corollary}
\theoremstyle{definition}
\newtheorem{definition}{Definition}[section]
\newtheorem{example}{Example}[section]

\newtheorem{notation}{Notation}[section]

\usepackage{algorithm}
\usepackage{algpseudocode}

\usepackage{tikz}

\algnewcommand\algorithmicinput{\textbf{Input:}}
\algnewcommand\INPUT{\item[\algorithmicinput]}
\algnewcommand\algorithmicoutput{\textbf{Output:}}
\algnewcommand\OUTPUT{\item[\algorithmicoutput]}

\algnewcommand\algorithmicproc{\textbf{Procedure:}}
\algnewcommand\PROCEDURE{\item[\algorithmicproc]}
\algnewcommand\algorithmiccomplexity{\textbf{Complexity:}}
\algnewcommand\COMPLEXITY{\item[\algorithmiccomplexity]}

\newlength{\continueindent}
\usepackage{etoolbox}
\makeatletter
\newcommand*{\ALG@customparshape}{\parshape 2 \leftmargin \linewidth \dimexpr\ALG@tlm+\continueindent\relax \dimexpr\linewidth+\leftmargin-\ALG@tlm-\continueindent\relax}
\apptocmd{\ALG@beginblock}{\ALG@customparshape}{}{\errmessage{failed to patch}}
\makeatother

\theoremstyle{remark}
\newtheorem{remark}{Remark}

\theoremstyle{observation}

\definecolor{DarkBlue}{rgb}{0,0.1,0.55}

\numberwithin{equation}{section}

\newcommand {\hide}[1]{}
 \newcommand {\sign} {\mbox{\bf sign}}

\newcommand {\junk}[1]{}

\newcommand {\R} {\mathrm{R}}

\newcommand {\D}     {\mbox{\rm D}}

\newcommand {\C}     {\mathrm{C}}

\newcommand {\Z}  {\mathbb{Z}}
 \newcommand {\N}         {\mathbb{N}}

\newcommand {\kk}         {\mathbf{k}}

\newcommand {\RR} {{\mathcal R}}

\newcommand {\la}   {{\langle}}
\newcommand {\ra}   {{\rangle}}
\newcommand {\eps} {{\varepsilon}}
\newcommand {\E} {{\rm ext}}

\newcommand {\Sign}      {\mbox{\rm Sign}}

\newcommand {\PP}     {\mathbb{P}} 

\newcommand{\card}{\mathrm{card}}

\def\addots{\mathinner{\mkern1mu
\raise1pt\vbox{\kern7pt\hbox{.}}
\mkern2mu\raise4pt\hbox{.}\mkern2mu
\raise7pt\hbox{.}\mkern1mu}}

\newcommand{\HH}  {\mbox{\rm H}}

\newcommand{\Hom}{\mathrm{Hom}}

\newcommand{\x}{\mathbf{x}}

\newcommand{\y}{\mathbf{y}}

\newcommand{\Cc}{\mathrm{Cc}}

\newcommand{\Pairs}{\mathrm{pairs}}

\newcommand{\Nerve}{\mathcal{N}}

\newcommand{\Simp}{\mathbf{Simp}}

\newcommand{\Vect}{\mathrm{Vect}}

\newcommand{\Cov}{\mathrm{Cov}}

\newcommand{\SAcat}{\mathbf{SA}}

\newcommand{\sk}{\mathrm{sk}}

\newcommand{\Elim}{\mathrm{Elim}}
\newcommand{\BElim}{\mathrm{BElim}}
\newcommand{\Trunc}{\mathrm{Trunc}}
\begin{document}
\title[Complexity and speed of semi-algebraic multi-persistence]
{
Complexity and speed of semi-algebraic multi-persistence
}
\author{Arindam Banerjee}
\address{Department of Mathematics, 
IIT Kharagpur, Kharagpur, India.}
\email{123.arindam@gmail.com}

\author{Saugata Basu}
\address{Department of Mathematics,
Purdue University, West Lafayette, IN 47906, U.S.A.}
\email{sbasu@math.purdue.edu}

\date{\textbf{\today}}
\thanks{
Banerjee was partially supported by SERB Start Up Research Grant, IIT Kharagpur Faculty Start Up Research Grant and CPDA of IIT Kharagpur.
Basu was  partially supported by NSF grant
CCF-1910441. }
\begin{abstract}
Let $\R$ be a real closed field, $S \subset \R^n$ a closed and bounded semi-algebraic set, and $\mathbf{f}=(f_1,\ldots,f_p):S \rightarrow \R^p$ a continuous semi-algebraic map inducing a $p$-parameter semi-algebraic filtration by sublevel sets. 
We introduce a barcode invariant for such filtrations that directly extends the classical ($p=1$) barcode. 
After 
scaling of the parameter space, in each homological degree $\ell$ the invariant is encoded by a $\Z_{\ge 0}$-valued
function
\[
\mu_\ell(S,\mathbf{f}):\ 
\Big(({-}1,1)^p\times(({-}1,1)^p \cup\{(1,\ldots,1)\}) \Big)\ \cap\ \{(\mathbf a,\mathbf b)\mid \mathbf a\preceq \mathbf b\}
\ \longrightarrow\ \Z_{\ge 0},
\]
where $\preceq$ denotes the product order on $\R^p$.
We prove that $\mu_\ell(S,\mathbf{f})$ is semi-algebraically constructible and establish a singly exponential upper bound on its description complexity. 
Moreover, we give a singly exponential-time algorithm to compute $\mu_\ell(S,\mathbf{f})$, extending to arbitrary $p$ the corresponding result for $p=1$ in \cite{Basu-Karisani-2}. 
Finally, for semi-algebraic filtrations of bounded description complexity we bound the number of equivalence classes of finite poset modules realizable in this way, yielding a tight analogue of ``speed'' bounds for algebraically defined graph classes.
\end{abstract}
\maketitle

\tableofcontents

\section{Introduction}
Persistent homology theory is now a well established sub-field of applied topology \cite{Weinberger_survey,Ghrist,dey2022computational}. Initially persistent homology modules were associated to filtrations of spaces by the sub-level sets of single functions. More recently, simultaneous filtrations by multiple functions have been studied driven by applications \cite{CZ2007,CSZ2010,BOO2022,DX2022,DKM2024}. 
More generally, persistent homology theory over arbitrary posets have been developed
extensively in \cite{miller2020homological}.
Study of persistent homology restricted to tame spaces -- such 
as the category of semi-algebraic sets and maps -- is of much recent origin \cite{KS2018,Miller2023,Basu-Karisani-2} -- and is the topic of this paper.
While barcodes of filtrations of finite simplicial complexes are standard  objects in persistent homology theory, their definition in the case of continuous filtrations (for example, in the semi-algebraic case) require care. A precise definition 
of semi-algebraic barcodes for semi-algebraic filtrations appears in 
\cite{Basu-Karisani-2}, where an algorithm with singly exponential complexity is also given for computing it. 

\medskip
In this paper we study, from a quantitative and algorithmic viewpoint, persistent homology modules arising from multi-parameter filtrations of semi-algebraic sets defined by several continuous semi-algebraic functions. 
We introduce a multi-parameter generalization of barcodes for semi-algebraic filtrations and give an explicit (symbolic) algorithm 
to compute it, together with singly exponential complexity bounds\footnote{see Section~\ref{subsec:complexity:algorithms} for the model of computation and definition of complexity}; this recovers fully the results proved in \cite{Basu-Karisani-2} in the one-parameter case. 
More precisely, for each fixed number of parameters $p$
and fixed homological degree $\ell$,
our results show that the barcode function of a semi-algebraic multi-filtration can be computed with \emph{singly exponential} complexity (rather than the doubly exponential bounds obtained via general triangulation methods). This adds multi-parameter persistent homology in fixed degree to the growing list of semi-algebraic algorithmic tasks admitting singly exponential algorithms — such as computing the dimension and Euler–Poincar\'e characteristic, computing Betti numbers in fixed degrees, and performing quantifier elimination with a fixed number of quantifier alternations.

Our results can be viewed as quantitative and algorithmic counterparts to more abstract treatments of multi-parameter persistence in tame settings --  notably those of Kashiwara--Schapira \cite{KS2018}, Curry--Patel \cite{Curry-Patel}, Miller \cite{Miller2023} and Berkouk \cite{Berkouk2023} -- thus establishing a bridge between these abstract works and algorithmic real algebraic geometry (as developed for example in the book \cite{BPRbook2}). 

\medskip
We also define a notion of \emph{speed} for families of semi-algebraically defined poset modules and, using our quantitative bounds, prove a sharp upper bound on this speed that parallels the classical graph-theoretic notion studied in \cite{Sauermann}. To our knowledge this establishes for the first time a connection between persistent homology theory and extremal graph theory. It also parallels several other examples where algorithmic results in real algebraic geometry lead to non-trivial upper bounds on some quantity of 
interest (see Remark~\ref{rem:byproduct} later).
\medskip

We now define semi-algebraic filtrations and barcodes more precisely.
Let $\R$ be a real closed field, fixed for the rest of the paper. 
We study semi-algebraic multi-filtrations defined by sublevel sets of continuous semi-algebraic maps
$\mathbf f:S\to \R^p$ with $p\ge 1$. 
We first extend to this multi-parameter setting the notion of semi-algebraic barcodes (in each homological degree $\ell\ge 0$) introduced in \cite{Basu-Karisani-2} for the one-parameter case $p=1$. 
We then give, for each fixed $\ell$, an algorithm with singly exponential complexity for computing these invariants. 
In our framework the degree-$\ell$ barcode is encoded by a $\Z_{\ge 0}$-valued semi-algebraically constructible function whose domain is a certain semi-algebraic subset of $\R^p\times \R^p$ (see Definition~\ref{def:sa-mp-barcode-mu}).

\medskip
In order to motivate the definition of barcodes in the multi-parameter case
we first recall below the definition of semi-algebraic barcodes in the one parameter case from \cite{Basu-Karisani-2}.

\subsection{One-parameter case}
\label{subsec:intro:one-parameter}

We fix a field $\kk$ and let $\Vect_{\kk}$ denote the category of $\kk$-vector spaces and $\kk$-linear maps.
For each $\ell\ge 0$ we write
\[
\HH_\ell(\cdot)=\HH_\ell(\cdot;\kk)
\]
for the $\ell$-th homology functor from the category of closed and bounded semi-algebraic sets (and continuous semi-algebraic maps) to $\Vect_{\kk}$
(see \cite[Chapter~6]{BPRbook2} for the definition over an arbitrary real closed field).

\medskip
Now let $S\subset \R^n$ be a closed and bounded semi-algebraic set and let $f:S\to\R$ be a continuous semi-algebraic map.
Then $f(S)$ is closed and bounded. 

\medskip
We recall the definition of the semi-algebraic barcode from \cite{Basu-Karisani-2}, with a minor normalization that will be convenient in the multi-parameter setting.
Instead of letting the parameter range be all of $\R$, we work over $(-1,1]$ and interpret the endpoint $1$ as the analogue of ``$+\infty$'':
in the case, $f(S) \subset (-1,1)$, 
\footnote{
This can always be achieved after rescaling the parameter (i.e.\ replacing $f$ by a suitable scalar multiple),
and this entails no loss of topological information.} 
the
semi-infinite bars in the usual barcode correspond bijectively to bars ending at $1$ in this normalized version, so 
no information is lost in this case. 

\medskip
The advantage of this convention will become clear when we define multi-parameter barcodes, where a direct analogue of ``semi-infinite'' bars requires additional care.

\medskip
Following \cite{Basu-Karisani-2}, we denote by
\[
\mathcal F \;=\; \bigl(S_{f\le t}\bigr)_{t\in(-1,1]}
\qquad\text{where}\qquad
S_{f\le t}=\{x\in S \mid f(x)\le t\},
\]
the filtration of $S$ by sublevel sets of $f$.
For $s\le t$ and $\ell\ge 0$, let
\[
\iota_\ell^{s,t}:\HH_\ell(S_{f\le s}) \longrightarrow \HH_\ell(S_{f\le t})
\]
be the homomorphism induced by the inclusion $S_{f\le s}\hookrightarrow S_{f\le t}$.
The family $\bigl(\iota_\ell^{s,t}\bigr)_{s\le t}$ defines a functor
\[
\mathbf P_{S,f,\ell}:\bigl((-1,1],\le\bigr)\longrightarrow \Vect_{\kk},
\qquad
\mathbf P_{S,f,\ell}(t)=\HH_\ell(S_{f\le t}),
\]
and we set
\[
\HH_\ell^{s,t}(\mathcal F)\;=\;\mathrm{Im}(\iota_\ell^{s,t}).
\]

\medskip
We next define certain subspaces of $\mathbf P_{S,f,\ell}(s)=\HH_\ell(S_{f\le s})$.

\begin{definition}[Subspaces associated to the filtration $\mathcal F$]
\label{def:barcode}
Let $\ell\ge 0$ and let $s\le t$ with $s,t\in(-1,1)$.
Define
\begin{align*}
M^{s,t}_\ell(\mathcal F)
&=\sum_{-1\le s'<s}\bigl(\iota_\ell^{s,t}\bigr)^{-1}\!\bigl(\HH_\ell^{s',t}(\mathcal F)\bigr),\\
N^{s,t}_\ell(\mathcal F)
&=\sum_{-1\le s'<s\le t'<t}\bigl(\iota_\ell^{s,t'}\bigr)^{-1}\!\bigl(\HH_\ell^{s',t'}(\mathcal F)\bigr).
\end{align*}
For $s\in(-1,1)$ we also set
\[
M^{s,1}_\ell(\mathcal F)\;=\;\sum_{s\le t'\le 1} M^{s,t'}_\ell(\mathcal F).
\]
\end{definition}

For $s\le t$ with $s,t\in(-1,1)$ and $\ell\ge 0$, define the quotient space
\[
Q^{s,t}_\ell(\mathcal F)\;=\;M^{s,t}_\ell(\mathcal F)\big/ N^{s,t}_\ell(\mathcal F),
\]
and for $s\in(-1,1)$ define
\[
Q^{s,1}_\ell(\mathcal F)\;=\;\HH_\ell(S_{f\le s})\big/ M^{s,1}_\ell(\mathcal F).
\]
Finally, for $s\in(-1,1)$ and $t\in(-1,1]$ with $s\le t$, set
\[
\mu^{s,t}_\ell(\mathcal F)\;=\;\dim_{\kk} Q^{s,t}_\ell(\mathcal F).
\]
In persistent homology, $\mu^{s,t}_\ell(\mathcal F)$ is interpreted as the multiplicity of degree-$\ell$ classes born at time $s$ and dying at time $t$ (with $t=1$ corresponding to classes that never die).

\medskip
We thus obtain a function
\begin{equation}
\label{eqn:barcode-sa-mu}
\mu_\ell(S,f):
\bigl(({-}1,1)\times({-}1,1]\bigr)\cap\{(s,t)\mid s\le t\}\longrightarrow \Z_{\ge 0},
\qquad
\mu_\ell(S,f)(s,t)=\mu^{s,t}_\ell(\mathcal F),
\end{equation}
which we call the \emph{degree-$\ell$ barcode} of the filtration $(S,f)$.

\begin{remark}
\label{rem:barcode-sa-mu}
The function $\mu_\ell(S,f)$ is often referred to as the \emph{degree-$\ell$ persistence diagram} of $(S,f)$.
We retain the terminology \emph{barcode} to remain consistent with \cite{Basu-Karisani-2}.
\end{remark}

It is shown in \cite{Basu-Karisani-2} that the function $\mu_\ell(S,f)$ is semi-algebraically constructible (see Definition~\ref{def:sa-constr-functions}); moreover, it has finite support of singly exponential size, and \emph{loc.\ cit.}\ gives an algorithm with singly exponential complexity for computing it.
Our goal in this paper is to extend these constructibility and effectivity results to the multi-parameter setting.

\subsection{Multi-parameter case}
\label{subsec:intro:mp}
In order to generalize the notion of barcodes to the multi-parameter
semi-algebraic setting it is convenient to first introduce some categorical language.

\begin{definition}[Poset modules]
\label{def:poset-module}
    Given a partially ordered set  (poset) $(\PP,\preceq)$, a
\emph{poset module over $\PP$} is a functor:
\[
\mathbf{P}: (\PP,\preceq) \rightarrow \Vect_{\mathbf{k}},
\]
where  we consider $(\PP,\preceq)$ as a poset category -- namely, the category whose objects are elements $p \in \PP$, and the set of morphisms $\Hom_{(\PP,\preceq )}(p,p')$ is empty if $p \not\preceq p'$ and a singleton set, 
which will denote by ``$p\preceq p'$'', otherwise.

We will denote by $\Pairs(\PP) = \{(p,p') \in \PP \times \PP \; \mid \; p \preceq p'\}$. 
\end{definition}

\begin{definition}[Isomorphism of poset modules]
\label{def:poset-module-isomorphism}
We say that two poset modules $\mathbf{P},\mathbf{P'}:\PP \rightarrow \Vect_{\mathbf{k}}$ are isomorphic if there exists a natural transformations $F,G$ as depicted below which are inverses of each other:
\[
\xymatrix{
\PP \ar@/^1pc/[rr]^{\mathbf{P}} \ar@/_1pc/[rr] _{\mathbf{P}'}& F \Uparrow  & \Vect_{\mathbf{k}},
}
\xymatrix{
\PP \ar@/^1pc/[rr]^{\mathbf{P}} \ar@/_1pc/[rr] _{\mathbf{P}'}& \Downarrow G & \Vect_{\mathbf{k}}.
}
\]
\end{definition}

We now generalize the barcode function $\mu_\ell$ from the one-parameter semi-algebraic setting to two broader contexts: first to arbitrary poset modules, and then to semi-algebraic multi-filtrations.

\medskip
The poset relevant for multi-parameter persistence is $(\R^p,\preceq)$ with the coordinatewise partial order. 
As in the classical one-parameter case ($p=1$), it is important to record homology classes that never die (corresponding to ``semi-infinite'' bars). 
To encode these classes while working over a bounded parameter space, we adopt a normalization analogous to restricting to $(-1,1]$ in the one-parameter case: we replace $\R^p$ by the bounded poset
\[
\PP_p := (-1,1)^p \cup \{(1,\ldots,1)\},
\]
ordered by the restriction of the product order, and interpret the distinguished maximal element $(1,\ldots,1)$ as the multi-parameter analogue of ``$+\infty$''.
\footnote{As in the case when $p=1$, when $\mathbf{f}(S) \subset (-1,1)^p$ no topological information will be lost by restricting the poset to $\PP_p$.
}
(When $p=1$ this recovers the convention of working on $(-1,1]$, with $1$ playing the role of the endpoint at infinity.)

\medskip
This leads to the following notation.

\begin{notation}
For a poset $(\PP,\preceq)$, let $\PP^{\max}\subset \PP$ be the set of maximal elements, and set $\PP^{o}=\PP\setminus \PP^{\max}$.
\end{notation}

\medskip
We next extend Definition~\ref{def:barcode} to an arbitrary poset module, replacing the linearly ordered poset $((-1,1],\le)$ by a general poset $(\PP,\preceq)$.

\begin{definition}[Barcode of a poset module]
\label{def:persistent}
Let $\mathbf P:(\PP,\preceq)\to \Vect_{\kk}$ be a functor.
For $(\mathbf s,\mathbf t)\in \Pairs(\PP)$, set
\[
\HH^{\mathbf s,\mathbf t}(\mathbf P)=\mathrm{Im}\bigl(\mathbf P(\mathbf s\preceq \mathbf t)\bigr)\subset \mathbf P(\mathbf t).
\]
For $(\mathbf s,\mathbf t)\in \Pairs(\PP^{o})$ define
\begin{align*}
M^{\mathbf s,\mathbf t}(\mathbf P)
&=\sum_{\mathbf s'\prec \mathbf s}\bigl(\mathbf P(\mathbf s\preceq \mathbf t)\bigr)^{-1}\!\bigl(\HH^{\mathbf s',\mathbf t}(\mathbf P)\bigr),\\
N^{\mathbf s,\mathbf t}(\mathbf P)
&=\sum_{\mathbf s'\prec \mathbf s\preceq \mathbf t'\prec \mathbf t}\bigl(\mathbf P(\mathbf s\preceq \mathbf t')\bigr)^{-1}\!\bigl(\HH^{\mathbf s',\mathbf t'}(\mathbf P)\bigr),\\
Q^{\mathbf s,\mathbf t}(\mathbf P)
&=M^{\mathbf s,\mathbf t}(\mathbf P)\big/ N^{\mathbf s,\mathbf t}(\mathbf P).
\end{align*}
For $(\mathbf s,\mathbf t)\in (\PP^{o}\times \PP^{\max})\cap \Pairs(\PP)$ set
\[
M^{\mathbf s,\mathbf t}(\mathbf P)=\sum_{\mathbf s\preceq \mathbf t'\preceq \mathbf t} M^{\mathbf s,\mathbf t'}(\mathbf P),
\qquad
Q^{\mathbf s,\mathbf t}(\mathbf P)=\mathbf P(\mathbf s)\big/ M^{\mathbf s,\mathbf t}(\mathbf P).
\]
For $(\mathbf s,\mathbf t)\in (\PP^{o}\times \PP)\cap \Pairs(\PP)$ define
\begin{equation}
\label{eqn:def:barcode:multiplicity}
\mu^{\mathbf s,\mathbf t}(\mathbf P)=\dim Q^{\mathbf s,\mathbf t}(\mathbf P).
\end{equation}
We call the function
\[
\mu(\mathbf P):(\PP^{o}\times \PP)\cap \Pairs(\PP)\to \Z_{\ge 0},
\qquad
(\mathbf s,\mathbf t)\mapsto \mu^{\mathbf s,\mathbf t}(\mathbf P),
\]
the \emph{barcode} of $\mathbf P$.
\end{definition}

\medskip
We now specialize to poset modules arising from semi-algebraic multi-filtrations.

\begin{definition}[Multi-persistence modules induced by a semi-algebraic multi-filtration]
\label{def:semi-algebraic-mp-poset-module}
Let $S\subset \R^n$ be a closed and bounded semi-algebraic set and let $\mathbf f:S\to \R^p$ be a continuous semi-algebraic map.
For $\mathbf y\in \R^p$ set
\[
S_{\mathbf f\preceq \mathbf y}=\{x\in S \mid \mathbf f(x)\preceq \mathbf y\}.
\]
For $\ell\ge 0$ we define the functor
\[
\mathbf P_{S,\mathbf f,\ell}:(\R^p,\preceq)\to \Vect_{\kk},
\]
called the \emph{$\ell$-th multi-persistence module of $(S,\mathbf f)$}, by
\[
\mathbf P_{S,\mathbf f,\ell}(\mathbf y)=\HH_\ell(S_{\mathbf f\preceq \mathbf y}),
\qquad
\mathbf P_{S,\mathbf f,\ell}(\mathbf y\preceq \mathbf y')=\iota^{\mathbf y,\mathbf y'}_\ell,
\]
where $\iota^{\mathbf y,\mathbf y'}_\ell$ is induced by the inclusion
$S_{\mathbf f\preceq \mathbf y}\hookrightarrow S_{\mathbf f\preceq \mathbf y'}$.
\end{definition}

\medskip
To implement the normalization discussed above, we introduce the bounded parameter poset:
\begin{notation}
\label{not:PP}
Let $(\PP_p,\preceq)$ denote the poset
\[
\PP_p=\bigl((-1,1)^p \cup \{(1,\ldots,1)\}\bigr),
\]
equipped with the restriction of the product order on $\R^p$.
\end{notation}
Note that $\PP_p^{\max}=\{(1,\ldots,1)\}$ and $\PP_p^{o}=(-1,1)^p$.

\begin{definition}[Barcode of a semi-algebraic multi-filtration]
\label{def:sa-mp-barcode-mu}
For $\ell\ge 0$ we define
\[
\mu_\ell(S,\mathbf f)
\;=\;
\mu\!\left(\mathbf P_{S,\mathbf f,\ell}\!\restriction_{\PP_p}\right)
:\; (\PP_p^{o}\times \PP_p)\cap \Pairs(\PP_p)\longrightarrow \Z_{\ge 0},
\]
and call $\mu_\ell(S,\mathbf f)$ the \emph{degree-$\ell$ barcode} of the multi-filtration $(S,\mathbf f)$.
\end{definition}

\begin{remark}
Definition~\ref{def:sa-mp-barcode-mu} reduces to the one-parameter definition \eqref{eqn:barcode-sa-mu} when $p=1$ (under the normalization convention identifying $1$ with $+\infty$).
\end{remark}

The key algorithmic result of the paper can be summarized informally as follows; precise statements (together with the relevant complexity notions for constructible functions and for our algorithmic model) appear later as Theorems~\ref{thm:main:mu} and \ref{thm:alg:main:mu}.

\begin{theorem*}[Informal statement of the main algorithmic result]
\label{thm:informal:main:mu}
Fix $\ell\ge 0$ and $p\ge 1$. Let $S$ be a closed and bounded semi-algebraic set and let $\mathbf f:S\to \R^p$ be a continuous semi-algebraic map. Then:
\begin{enumerate}[(a)]
\item the barcode function $\mu_\ell(S,\mathbf f)$ is semi-algebraically constructible, with singly exponential description complexity;
\item there is an algorithm with singly exponential complexity that computes (a semi-algebraic description of) $\mu_\ell(S,\mathbf f)$.
\end{enumerate}
\end{theorem*}

\begin{remark}
\label{rem:finite}
In Definition~\ref{def:persistent}, the subspaces $M^{\mathbf s,\mathbf t}(\mathbf P)$ and $N^{\mathbf s,\mathbf t}(\mathbf P)$ are defined as sums indexed by (potentially) infinitely many parameters $\mathbf s',\mathbf t'$.
For general poset modules this may indeed involve infinitely many distinct subspaces. 
However, for poset modules arising from semi-algebraic filtrations, only finitely many distinct subspaces occur, and this finiteness is a crucial input in the proof of Theorem~\ref{thm:informal:main:mu}.
It ultimately follows from the semi-algebraic constructibility of the underlying multi-persistence modules $\mathbf P_{S,\mathbf f,\ell}$ and their normalized restrictions $\mathbf P_{S,\mathbf f,\ell}\!\restriction_{\PP_p}$.
We record the latter fact informally here; a precise version (after defining constructibility and its complexity for poset modules) appears as Theorem~\ref{thm:main}.
\end{remark}

\begin{theorem*}[Informal statement on constructibility of the induced poset modules]
\label{thm:informal:main}
The poset modules $\mathbf P_{S,\mathbf f,\ell}$ and $\mathbf P_{S,\mathbf f,\ell}\!\restriction_{\PP_p}$ are semi-algebraically constructible with singly exponential complexity. Moreover, each can be constructed by an algorithm with singly exponential complexity.
\end{theorem*}

The following example illustrates the functions $\mu_\ell(S,\mathbf f)$ in a simple case.

\begin{example}
Let $S\subset \R^2$ be the circle of radius $1/2$ centered at $\mathbf 0$, and let $\mathbf f=(X_1,X_2)$ be the coordinate map. Then $\mathbf f(S)\subset (-1,1)^2$.
One checks that $\mu_0(S,\mathbf f)$ and $\mu_1(S,\mathbf f)$ (see Figure~\ref{fig:mu0-mu1-support}) are given by:
\[
\mu_0(S,\mathbf f)(\mathbf s,\mathbf t)=
\begin{cases}
1 & \text{if $\mathbf s\in \R_{\le 0}^2$, $\|\mathbf s\|^2=1/4$, and $\mathbf t=(1,1)$,}\\
0 & \text{otherwise,}
\end{cases}
\]
\[
\mu_1(S,\mathbf f)(\mathbf s,\mathbf t)=
\begin{cases}
1 & \text{if $\mathbf s=(1/2,1/2)$ and $\mathbf t=(1,1)$,}\\
0 & \text{otherwise.}
\end{cases}
\]
\end{example}

\begin{figure}[t]
\centering
\begin{tikzpicture}[scale=4,>=latex]

\begin{scope}
  \draw[->] (-0.75,0) -- (0.75,0) node[below right] {$s_1$};
  \draw[->] (0,-0.75) -- (0,0.75) node[above left] {$s_2$};

  \fill[black!4] (-0.75,0) rectangle (0,-0.75);

  \draw[thick] (-0.5,0) arc (180:270:0.5);

  \fill (-0.5,0) circle (0.012);
  \fill (0,-0.5) circle (0.012);

  \node[align=center] at (0,-0.97)
    {$\mathrm{supp}\,\mu_0(\,\cdot\,,(1,1))$\\[1pt]
     $\{\|s\|=\tfrac12,\ s_1\le 0,\ s_2\le 0\}$};

  \node[font=\small] at (0,1.08) {$\mu_0(S,\mathbf f)$};
\end{scope}

\begin{scope}[xshift=1.0cm]
  \draw[->] (-0.1,0) -- (1.05,0) node[below right] {$s_1$};
  \draw[->] (0,-0.1) -- (0,1.05) node[above left] {$s_2$};

  \fill (0.5,0.5) circle (0.02);
  \draw[dashed] (0.5,0) -- (0.5,0.5);
  \draw[dashed] (0,0.5) -- (0.5,0.5);

  \node[anchor=west] at (0.53,0.52) {$(\tfrac12,\tfrac12)$};

  \node[align=center] at (0.55,-0.20)
    {\vspace{.1in} $\mathrm{supp}\,\mu_1(\,\cdot\,,(1,1))=\{(\tfrac12,\tfrac12)\}$};

  \node[font=\small] at (0.52,1.08) {$\mu_1(S,\mathbf f)$};
\end{scope}

\end{tikzpicture}
\vspace{.2in}
\caption{Supports of $\mu_0(S,\mathbf f)$ and $\mu_1(S,\mathbf f)$ in the $\mathbf s$-plane (the value is $1$ only when $\mathbf t=(1,1)$).}
\label{fig:mu0-mu1-support}
\end{figure}

In the next two subsections we make precise the notions of complexity used throughout the paper: for semi-algebraically constructible functions and poset modules, and for the algorithms that we employ. We also relate these definitions to similar (but not identical) notions appearing in earlier work.

\subsection{Semi-algebraically constructible functions and poset modules, and their complexity}
\label{subsec:complexity:functions}

\begin{definition}
\label{def:sa-constr-functions}
Let $S$ be a semi-algebraic set. A function
\[
F:S\to \kk
\]
is \emph{semi-algebraically constructible} if it is a $\kk$-linear combination of characteristic functions of finitely many semi-algebraic subsets of $S$.
More generally, a map $F=(f_1,\ldots,f_M):S\to \kk^M$ is semi-algebraically constructible if each coordinate function $f_i$ is.
\end{definition}

Equivalently, for every semi-algebraically constructible $F:S\to\kk^M$ there exists a finite partition of $S$ into semi-algebraic subsets on each of which $F$ is constant. We say that such a partition is \emph{subordinate to $F$}. Any refinement of a subordinate partition is again subordinate to $F$.

\medskip
We will work with a particularly convenient class of semi-algebraic partitions defined via sign conditions.

\begin{notation}[Sign conditions, realizations, and $\Cc(\cdot)$]
\label{not:sign-condition}
Let $\mathcal P\subset \R[X_1,\ldots,X_n]$ be finite. A \emph{sign condition on $\mathcal P$} is a map $\sigma\in\{0,1,-1\}^{\mathcal P}$.
Its realization is
\[
\RR(\sigma)=\{\x\in \R^n \mid \sign(P(\x))=\sigma(P)\ \text{for all }P\in \mathcal P\}.
\]
We denote by
\[
\Sign(\mathcal P)=\{\sigma\in \{0,1,-1\}^{\mathcal P}\mid \RR(\sigma)\neq \emptyset\}
\]
the set of \emph{realizable} sign conditions of $\mathcal P$.
For $\sigma\in \Sign(\mathcal P)$, let $\Cc(\sigma)$ be the set of semi-algebraically connected components of $\RR(\sigma)$, and set
\[
\Cc(\mathcal P)=\bigcup_{\sigma\in \Sign(\mathcal P)} \Cc(\sigma).
\]
Thus $\bigl(C\bigr)_{C\in \Cc(\mathcal P)}$ is a finite partition of $\R^n$ into nonempty, locally closed semi-algebraic sets.
\end{notation}

\begin{notation}
For a finite $\mathcal P\subset \R[X_1,\ldots,X_n]$, we write
\[
\deg(\mathcal P)=\max_{P\in \mathcal P}\deg(P).
\]
\end{notation}

\medskip
We now define the complexity of a semi-algebraically constructible function in terms of the complexity of a sign-condition partition subordinate to it.

\begin{definition}[Complexity of semi-algebraically constructible functions]
\label{def:constructible-function-complexity}
Let $F:S\to \kk^M$ be semi-algebraically constructible, with $S\subset \R^n$ semi-algebraic.
We say that $F$ has \emph{complexity $\le D$} if there exists a finite set $\mathcal P\subset \R[X_1,\ldots,X_n]$ such that
\begin{enumerate}[(i)]
\item the partition $\bigl(C\bigr)_{C\in \Cc(\mathcal P),\, C\subset S}$ of $S$ is subordinate to $F$, and
\item $\card(\mathcal P)\cdot \deg(\mathcal P)\le D$.
\end{enumerate}
\end{definition}

\begin{remark}[Alternative complexity conventions]
\label{rem:alt-complexity}
One could alternatively measure complexity by requiring only that the coarser partition by realizations
\[
\bigl(\RR(\sigma)\bigr)_{\sigma\in \Sign(\mathcal P),\, \RR(\sigma)\subset S}
\]
be subordinate to $F$, with the same bound $\card(\mathcal P)\cdot \deg(\mathcal P)\le D$.
The two conventions are polynomially related: if $\card(\mathcal P)\le s$ and $\deg(\mathcal P)\le d$, then every semi-algebraically connected component of a $\mathcal P$-semi-algebraic set is a $\mathcal Q$-semi-algebraic set for some finite $\mathcal Q\subset \R[X_1,\ldots,X_n]$ with
\[
\card(\mathcal Q)\le s^{n} d^{O(n^4)}
\qquad\text{and}\qquad
\deg(\mathcal Q)\le d^{O(n^3)}
\]
(see \cite[Theorem~16.11]{BPRbook2}).
Consequently, switching between these conventions changes our bounds by at most a factor of $D^{\,n^{O(1)}}$, and therefore does not affect the asymptotic singly exponential estimates proved in this paper (e.g.\ Theorem~\ref{thm:main}).

\medskip
Semi-algebraically constructible functions have also been studied from the viewpoint of computational complexity in \cite{Basu2015}. There, complexity is defined (in terms of formula size) via subordinate partitions as well \cite[Definition~2.16]{Basu2015}, but in a more refined way aimed at developing complexity classes of constructible functions (and sheaves). Since this is not our goal here, we do not reproduce that definition; we note, however, that our upper bounds (for instance in Theorem~\ref{thm:main}) remain valid under the convention of \cite{Basu2015}.
\end{remark}

\medskip
We will use the following notation.

\begin{notation}
\label{not:matrix}
For $0\le p,q\le N$, let
\[
\Trunc_{p,q}^{N}:\kk^{N\times N}\to \kk^{p\times q}
\]
denote the truncation map that extracts the upper-left $p\times q$ block of an $N\times N$ matrix.
For $M\in \kk^{p\times q}$, let $L_M:\kk^{q}\to \kk^{p}$ denote the linear map $\x\mapsto M\x$.
\end{notation}

\begin{definition}[Semi-algebraically constructible poset modules and their complexity]
\label{def:constructible-poset-module}
Let $\mathbf P:(\R^p,\preceq)\to \Vect_{\kk}$ (resp.\ $\mathbf P:(\PP_p,\preceq)\to \Vect_{\kk}$) be a poset module.
We say that $\mathbf P$ is \emph{semi-algebraically constructible} if there exist $N\ge 0$ and a semi-algebraically constructible function
\[
F:\Pairs(\R^p)\to \kk^{N\times N}
\qquad\text{(resp.\ $F:\Pairs(\PP_p)\to \kk^{N\times N}$)}
\]
such that $\mathbf P$ is isomorphic (see Definition~\ref{def:poset-module-isomorphism}) to the poset module
$\widetilde{\mathbf P}:(\R^p,\preceq)\to \Vect_{\kk}$ (resp.\ $\widetilde{\mathbf P}:(\PP_p,\preceq)\to \Vect_{\kk}$) defined by
\begin{align*}
\widetilde{\mathbf P}(\y) &= \kk^{\dim \mathbf P(\y)},\\
\widetilde{\mathbf P}(\y\preceq \y') &=
L_{\Trunc^{N}_{\dim \mathbf P(\y'),\,\dim \mathbf P(\y)}\bigl(F(\y,\y')\bigr)}.
\end{align*}
We call such an $F$ an \emph{associated constructible function} for $\mathbf P$.
We say that $\mathbf P$ has \emph{complexity $\le D$} if it admits an associated constructible function $F$ of complexity $\le D$ in the sense of Definition~\ref{def:constructible-function-complexity}.
\end{definition}

The following simple example illustrates Definitions~\ref{def:constructible-poset-module} and \ref{def:semi-algebraic-mp-poset-module}.
\begin{example}

\begin{figure}[t]
\centering
\begin{tikzpicture}[scale=1.75,>=latex]

\begin{scope}
  \draw[->] (-1.6,0) -- (1.6,0) node[below right] {$y_1$};
  \draw[->] (0,-1.6) -- (0,1.6) node[above left] {$y_2$};

  \fill[black!15] (0,0) circle (1);
  \draw[very thick] (0,0) circle (1);
  \node[font=\small] at (-0.55,0.55) {$S$};

  \draw[thin] (-1,-1.5) -- (-1,1.5);
  \draw[thin] (-1,0) -- (1.5,0);
  \draw[thin] (0,-1) -- (0,1.5);
  \draw[thin] (0,-1) -- (1.5,-1);


  \foreach \x in {-1,1} { \draw (\x,0.03)--(\x,-0.03) node[below] {\scriptsize $\x$}; }
  \foreach \y in {-1,1} { \draw (0.03,\y)--(-0.03,\y) node[left] {\scriptsize $\y$}; }
\end{scope}

\begin{scope}[xshift=3.6cm]
  \draw[->] (-1.6,0) -- (1.6,0) node[below right] {$y_1$};
  \draw[->] (0,-1.6) -- (0,1.6) node[above left] {$y_2$};

  \begin{scope}
    \clip (-1.5,-1.5) rectangle (1.5,1.5);
    \fill[black!18] (-1,0) rectangle (1.5,1.5);  
    \fill[black!18] (0,-1) rectangle (1.5,1.5);  
    \fill[black!18] (0,0) circle (1);            
  \end{scope}

  \draw[thick] (0,0) circle (1);
  \draw[thick] (-1,-1.5) -- (-1,1.5);
  \draw[thick] (-1,0) -- (1.5,0);
  \draw[thick] (0,-1) -- (0,1.5);
  \draw[thick] (0,-1) -- (1.5,-1);

  \draw[->,very thick] (-0.6,-0.2) -- (0.8,0.6);
  \node[anchor=west,font=\scriptsize] at (0.85,0.62) {$\mathbf y_2$};
  \node[anchor=west,font=\scriptsize] at (-0.75,-0.3) {$\mathbf y_1$};

  \node[align=left] at (0.07,-1.25) {\small
  $F(\mathbf y_1,\mathbf y_2)=1$ iff\\
  $\mathbf y_1,\mathbf y_2\in D$ and $\mathbf y_1\preceq \mathbf y_2$.};


\end{scope}

\end{tikzpicture}
\caption{Left: the unit disk $S$ (shaded). Right: the set $D=S\cup(\R_{\ge0}^2+(-1,0))\cup(\R_{\ge0}^2+(0,-1))$ (uniformly shaded), and $F$ is the indicator of $(D\times D)\cap \Pairs(\R^2)$.}
\label{fig:example}
\end{figure}


    Let $n=2,p = 2,\ell = 0$, and $S \subset \R^2$ be the closed unit disk defined by $X_1^2 + X_2^2 -1 \leq 0$, and 
    $f_1 = X_1,f_2 = X_2$. 
    Let $D =  S \cup (\R_{\geq 0}^2 + (-1,0))  \cup (\R_{\geq 0}^2 + (0,-1))$ (see Figure~\ref{fig:example}), 
and let $F: \Pairs(\R^2) \rightarrow \mathbf{k}$ be the constructible function defined as follows:
\begin{eqnarray*}
    F(\y_1,\y_2) &=& 1, \mbox{ if $(\y_1,\y_2) \in D \times D \cap \Pairs(\R^2)$,}\\
    &=& 0, \mbox{else}.
\end{eqnarray*}
It is easy to check that $F$ is associated to the poset module 
$\mathbf{P}_{S,\mathbf{f},0}$ in this example.
Moreover, 
the partition $(C)_{C\in \Cc(\mathcal{D}), C \subset \Pairs(\R^2)}$ is subordinate to $F$,
where 
\[
\mathcal{D} = \{ Y_1^2 + Y_2^2 -1, Y_1+1, Y_2+1, 
Y_1'^2 + Y_2'^2 -1, Y_1'+1, Y_2' +1, Y_1 - Y_1',Y_2 - Y_2'\}.
\]

It follows from Definitions~\ref{def:constructible-function-complexity}
, \ref{def:constructible-poset-module} and \ref{def:semi-algebraic-mp-poset-module}, 
that
the complexity of $F$, and hence also the complexity of $\mathbf{P}_{S,\mathbf{f},0}$, is
$\leq \card(\mathcal{D})\cdot \deg(\mathcal{D}) = 8 \cdot 2  = 16$.
\end{example}

\subsection{Model of computation and complexity of algorithms}
\label{subsec:complexity:algorithms}
We fix an ordered domain $\D\subset \R$ for the rest of the paper.
There are several possible computational models for algorithmic problems involving semi-algebraic sets (and several corresponding notions of ``algorithm'').
When $\R=\mathbb R$ and $\D=\mathbb Z$, one may work in the classical Turing model and measure bit-complexity.
Here we follow \cite{BPRbook2} and adopt a model that is meaningful over an arbitrary real closed field.
In the special case $\D=\mathbb Z$, our bounds imply corresponding bit-complexity estimates. The precise notion of complexity that we use is given below.

\begin{definition}[Complexity of algorithms]
\label{def:complexity}
Our algorithms take as input quantifier-free first-order formulas whose terms are polynomials with coefficients in an ordered domain $\D\subset \R$.
By the \emph{complexity} of an algorithm we mean the number of arithmetic operations and comparisons performed in $\D$.
When $\D=\R$, this coincides with the Blum--Shub--Smale notion of real-number complexity \cite{BSS}.
\footnote{When $\D=\Z$, one can bound the bit-complexity in terms of the bit-sizes of the input coefficients: it is polynomial in the input bit-size times the operation-count bound stated in the paper. We do not state bit-complexity separately.}
\end{definition}

\subsection{Some useful notation}
\begin{notation}[Realizations, $\mathcal P$- and $\mathcal P$-closed semi-algebraic sets]
For a finite set $\mathcal P\subset \R[X_1,\ldots,X_n]$, we call a quantifier-free first-order formula $\phi$ whose atoms are of the form
$P=0$, $P<0$, or $P>0$ with $P\in \mathcal P$ a \emph{$\mathcal P$-formula}.
For any semi-algebraic set $V\subset \R^n$, the realization of $\phi$ in $V$ is
\[
\RR(\phi,V)=\{\x\in V \mid \phi(\x)\},
\]
and we call $\RR(\phi,V)$ a \emph{$\mathcal P$-semi-algebraic subset of $V$}.
If $V=\R^n$ we write $\RR(\phi)=\RR(\phi,\R^n)$ and call $\RR(\phi)$ a $\mathcal P$-semi-algebraic set.

\medskip
We say that a quantifier-free formula $\phi$ is \emph{closed} if it is in disjunctive normal form, contains no negations, and has atoms of the form $P\ge 0$ or $P\le 0$.
If the set of polynomials appearing in a closed formula is contained in a finite set $\mathcal P$, we call $\phi$ a \emph{$\mathcal P$-closed formula}, and call $\RR(\phi,V)$ a \emph{$\mathcal P$-closed semi-algebraic subset of $V$}.

\medskip
Similarly, for $\mathcal Q\subset \R[X_1,\ldots,X_n,Y_1,\ldots,Y_p]$, we call a map $f:\R^n\to \R^p$ \emph{$\mathcal Q$-semi-algebraic} (resp.\ \emph{$\mathcal Q$-closed semi-algebraic}) if $\mathrm{graph}(f)\subset \R^n\times \R^p$ is a $\mathcal Q$-semi-algebraic (resp.\ $\mathcal Q$-closed semi-algebraic) set.
\end{notation}

\subsection{Quantitative results}
\label{subsec:main:quantitative}

\begin{theorem}[Constructibility of the induced multi-persistence modules]
\label{thm:main}
Let $S\subset \R^n$ be a bounded $\mathcal P$-closed semi-algebraic set, and let $\mathbf f:S\to \R^p$ be a $\mathcal Q$-closed continuous semi-algebraic map.
Set
\[
s=\card(\mathcal P)+\card(\mathcal Q),
\qquad
d=\max\bigl(\deg(\mathcal P),\deg(\mathcal Q)\bigr).
\]
Then, for every $\ell\ge 0$, the poset modules $\mathbf P_{S,\mathbf f,\ell}$ and $\mathbf P_{S,\mathbf f,\ell}\!\restriction_{\PP_p}$ are semi-algebraically constructible (in the sense of Definition~\ref{def:constructible-poset-module}) with complexity bounded by
\[
(s d)^{n^{O(\ell)}p^{O(1)}}.
\]
\end{theorem}

\begin{theorem}[Constructibility of the multi-parameter barcode function]
\label{thm:main:mu}
With notation as in Theorem~\ref{thm:main}, for every $\ell\ge 0$ the barcode function $\mu_\ell(S,\mathbf f)$ is semi-algebraically constructible (Definition~\ref{def:constructible-function-complexity}) with complexity bounded by
\[
(s d)^{n^{O(\ell)}p^{O(1)}}.
\]
\end{theorem}

\subsection{Algorithmic result}
\label{subsec:main:alg}

We now state the main algorithmic result. For each fixed $\ell\ge 0$, we give an algorithm with singly exponential complexity that computes the constructible function $\mu_\ell(S,\mathbf f)$ from semi-algebraic descriptions of $S$ and $\mathbf f$.

\begin{theorem}[Algorithm for computing barcode of a semi-algebraic multi-filtration]
\label{thm:alg:main:mu}
There exists an algorithm that takes as input:
\begin{enumerate}[(a)]
\item finite sets of polynomials
$\mathcal P\subset \D[X_1,\ldots,X_n]$ and
$\mathcal Q\subset \D[X_1,\ldots,X_n,Y_1,\ldots,Y_p]$;
\item a $\mathcal P$-closed formula $\Phi$ with $\RR(\Phi)=S \subset \R^n$;
\item a $\mathcal Q$-closed formula $\Psi$ with $\RR(\Psi)= \mathrm{graph}(\mathbf f)\subset \R^{n+p}$, where $\mathbf f:S\to \R^p$ is continuous and $\mathcal Q$-closed;
\end{enumerate}
and produces as output:
\begin{enumerate}[(a)]
    \item a finite set $\mathcal{D}' \subset \D[Y_1,\ldots,Y_p, Y_1',\ldots,Y_p']$, 
containing the polynomials $Y_i \pm 1,Y_i'\pm 1, Y_i - Y_i', 1 \leq i \leq p$; 
\item
for each 
\[
D\in \Cc(\mathcal D')
\quad\text{with}\quad
D\subset (\PP_p^{o}\times \PP_p)\cap \Pairs(\PP_p),
\] 
an integer
$\mu_\ell(D) \in \Z_{\geq 0}$.
\end{enumerate}

Each pair $(D,\mu_\ell(D))$ 
satisfies the property that for all 
$(\mathbf s,\mathbf t) \in D$
\[
\mu_\ell(S,\mathbf{f})(\mathbf s,\mathbf t) = 
\mu_\ell(D).
\]

Let $s=\card(\mathcal P)+\card(\mathcal Q)$ and $d=\max\bigl(\deg(\mathcal P),\deg(\mathcal Q)\bigr)$.
The complexity of the algorithm (Definition~\ref{def:complexity}), as well as the number of polynomials in $\mathcal D'$ and their degrees, are bounded by
\[
(s d)^{n^{O(\ell)}p^{O(1)}}.
\]
\end{theorem}

\subsection{Bounding speeds of semi-algebraic multi-persistence modules}
\label{subsec:main:speed}
Our next result concerns a notion of \emph{speed} for semi-algebraically defined families of multi-persistence modules that we introduce in this paper. 
To motivate it, we first recall the analogous and well-studied notion for semi-algebraic graph classes.

\subsubsection{Speed of semi-algebraic graphs}
Let $F\subset \R^p\times \R^p$ be a fixed semi-algebraic relation. 
A labelled (directed) graph $(V,E\subset V\times V)$ with vertex set $V=[N]=\{1,\ldots,N\}$ is called an \emph{$F$-graph} if there exist points
$\y_1,\ldots,\y_N\in \R^p$ such that
\[
(i,j)\in E \quad\Longleftrightarrow\quad (\y_i,\y_j)\in F.
\]
Families of $F$-graphs (with $F$ semi-algebraic) have been studied from many perspectives, including Ramsey-type problems and extremal questions (see, e.g., \cite{Alon-Pach-et-al,Fox-et-al}). 
Here we focus on their \emph{speed}.

\begin{definition}[Speed of an $F$-graph class]
For a semi-algebraic relation $F\subset \R^p\times \R^p$, let
\[
G_F(N)=\#\{\text{$F$-graphs on the labelled vertex set $[N]$}\}.
\]
We call $G_F:\N\to \N$ the \emph{speed} of the family of graphs defined by $F$.
\end{definition}

The following theorem (stated here in a slightly simplified form) is a key result in the theory of semi-algebraic graphs.

\begin{theorem*}[\cite{Sauermann}, Theorem~1.2]
For every semi-algebraic relation $F\subset \R^p\times \R^p$,
\[
G_F(N)\le N^{(1+o_F(1))\,pN},
\]
where the term $o_F(1)$ depends on $F$.
\footnote{Theorem~1.2 of \cite{Sauermann} is more general: it bounds the number of edge-labellings taking values in a fixed finite alphabet $\Lambda$. The statement above is the special case $\Lambda=\{0,1\}$.}
\end{theorem*}

There is a compelling analogy between semi-algebraically defined graph classes and semi-algebraically defined families of multi-persistence modules. 
This naturally leads to the question whether an analogue of Sauermann's speed bound holds for families of multi-persistence modules arising from semi-algebraic filtrations. 
We answer this question affirmatively.

\medskip
To formulate the speed question for poset modules, we first specify the relevant notion of equivalence.

\begin{definition}[Strong and weak equivalence of poset modules]
\label{def:poset:equivalence}
Let
\[
\mathbf P':(P',\preceq')\to \Vect_{\kk},
\qquad
\mathbf P'':(P'',\preceq'')\to \Vect_{\kk}
\]
be poset modules.
We say that $\mathbf P'$ and $\mathbf P''$ are \emph{strongly equivalent} if $(P',\preceq')=(P'',\preceq'')$ and $\mathbf P'\cong \mathbf P''$ as functors (see Definition~\ref{def:poset-module-isomorphism}).

We say that $\mathbf P'$ and $\mathbf P''$ are \emph{weakly equivalent} if there exists a poset isomorphism
\[
\phi:(P',\preceq')\to (P'',\preceq'')
\]
such that $\mathbf P'$ is strongly equivalent to $\mathbf P''\circ \phi$.
Clearly, strong equivalence implies weak equivalence, but not conversely.
\end{definition}

\begin{remark}
In this section our statements do not involve barcodes, and therefore we do not need to record ``infinite bars''. Accordingly, we work with the full poset $(\R^p,\preceq)$ rather than the normalized poset $\PP_p$.
\end{remark}

\medskip
Given a semi-algebraic multi-filtration, we will restrict it to finite parameter sets.

\begin{definition}[Finite restrictions of a multi-persistence module]
\label{def:poset-module-finite}
Let $S$ be a closed and bounded semi-algebraic set and let $\mathbf f:S\to \R^p$ be a continuous semi-algebraic map. Fix $\ell\ge 0$.
For a finite tuple $T=(\mathbf t_1,\ldots,\mathbf t_N)\in (\R^p)^N$, let $\preceq_T$ be the partial order on $[N]$ defined by
\[
j\preceq_T j' \quad\Longleftrightarrow\quad \mathbf t_j\preceq \mathbf t_{j'}.
\]
We denote by
\[
\mathbf P_{S,\mathbf f,T,\ell}:\bigl([N],\preceq_T\bigr)\to \Vect_{\kk}
\]
the resulting poset module given by
\[
\mathbf P_{S,\mathbf f,T,\ell}(j)=\HH_\ell\bigl(S_{\mathbf f\preceq \mathbf t_j}\bigr),
\]
with structure maps induced by inclusion whenever $j\preceq_T j'$.
\end{definition}

\begin{theorem}[Speed of semi-algebraic multi-persistence]
\label{thm:speed}
Let $S\subset \R^n$ be a bounded $\mathcal P$-closed semi-algebraic set and let $\mathbf f:S\to \R^p$ be a $\mathcal Q$-closed continuous semi-algebraic map.
Set
\[
s=\card(\mathcal P)+\card(\mathcal Q),
\qquad
d=\max\bigl(\deg(\mathcal P),\deg(\mathcal Q)\bigr).
\]
Then for every $\ell\ge 0$ and every $N\ge 1$, the number of \emph{strong} (and hence also \emph{weak}) equivalence classes among the finite poset modules
\[
\mathbf P_{S,\mathbf f,T,\ell},\qquad T\in (\R^p)^N,
\]
is bounded by
\begin{equation}
\label{eqn:thm:speed:1}
\sum_{j=1}^{pN} \binom{C\,N^{2}}{j}\, C^{pN}
\;=\;
N^{(1+o(1))\,pN},
\end{equation}
where $C=(sd)^{(np)^{O(\ell)}}$.
\end{theorem}

\begin{remark}
\label{rem:speed:o1}
In \eqref{eqn:thm:speed:1}, the term $o(1)$ depends on the parameters $n,p,\ell,s,$ and $d$ (equivalently, on $S$ and $\mathbf f$ through a bound on their description complexity).
\end{remark}

\begin{remark}[Tightness]
\label{rem:tight}
The bound \eqref{eqn:thm:speed:1} is close to best possible in general.

\smallskip\noindent
\emph{Strong equivalence.}
Take $p=1$, $\ell=0$, $S=[0,1]^N$, and $f=X_1$ (equivalently, the projection to the first coordinate).
For $T=(t_1,\ldots,t_N)\in [0,1]^N$ with pairwise distinct coordinates, the induced order $\preceq_T$ on $[N]$ is a total order determined by the relative order type of $(t_1,\ldots,t_N)$.
Different order types yield non-isomorphic posets on $[N]$, hence distinct strong equivalence classes of modules. 
Thus we obtain at least $N!=N^{(1-o(1))N}$ strong equivalence classes.

\smallskip\noindent
\emph{Weak equivalence.}
In the preceding example, all total orders on $[N]$ are isomorphic as posets, so the corresponding modules fall into a single weak equivalence class.
By contrast, already for $p=2$ one can obtain many weak equivalence classes: for $S=[0,1]^2$, $\mathbf f=(X_1,X_2)$, and $\ell=0$, one can show that the number of weak equivalence classes among $\mathbf P_{S,\mathbf f,T,0}$ with $T\in (\R^2)^N$ is at least $2^{N-1}$.

\smallskip\noindent
Finally, recall that the total number of labelled poset structures on $[N]$ is asymptotic to $2^{N^2/4+o(N^2)}$ \cite{KR70}. 
Theorem~\ref{thm:speed} implies that, for each fixed $p$, the number of posets realizable as induced subposets on an $N$-point subset of $\R^p$ is at most $N^{(1+o(1))pN}$.
\end{remark}

\begin{remark}[Comparison with the graph-theoretic speed bound]
\label{rem:comparison}
It is instructive to compare Theorem~\ref{thm:speed} with its graph-theoretic counterpart for semi-algebraic graph classes. 
An $F$-graph may contain directed cycles, whereas the underlying directed graph of a poset is acyclic; thus, at the level of underlying edge sets, there are far fewer labelled acyclic directed graphs than all labelled directed graphs on $[N]$. 
On the other hand, a poset module assigns to each comparable pair a linear map, so the data are richer than a mere edge set.
Indeed, we show later (see Example~\ref{eg:infinite}) that even for a fixed finite poset and a fixed upper bound on the dimensions of the vector spaces (already $\le 2$), the set of strong (and even weak) equivalence classes of poset modules can be infinite, because of the freedom in choosing the linear maps.
\end{remark}

We also have the following uniform version of Theorem~\ref{thm:speed}. 
Although the asymptotic form of the bound in $N$ is the same, the dependence of the implicit constants on the parameters $s,d,n,p$ is worse than in Theorem~\ref{thm:speed}.

\begin{theorem}[Uniform speed bound]
\label{thm:speed:uniform}
Fix integers $n,p\ge 1$, $\ell\ge 0$, and positive integers $s,d$.
For each $N\ge 1$, consider the collection of poset modules of the form
\[
\mathbf P_{S,\mathbf f,T,\ell},
\qquad T\in(\R^p)^N,
\]
where $S\subset \R^n$ is a bounded $\mathcal P$-closed semi-algebraic set and $\mathbf f:S\to \R^p$ is a $\mathcal Q$-closed continuous semi-algebraic map, for some finite sets
\[
\mathcal P\subset \R[X_1,\ldots,X_n]_{\le d},
\qquad
\mathcal Q\subset \R[X_1,\ldots,X_n,Y_1,\ldots,Y_p]_{\le d}
\]
satisfying $\card(\mathcal P)+\card(\mathcal Q)\le s$.
Then the number of \emph{strong} (and hence also weak)
equivalence classes in this collection is at most
\[
N^{(1+o(1))\,pN},
\]
where the $o(1)$ term (and the implied constants) may depend on $n,p,\ell,s,$ and $d$, but not on $N$.
\end{theorem}

\section{Background, significance and prior work}
\label{sec:background}

\subsection{Algorithmic semi-algebraic geometry}
The algorithmic computation of topological invariants of semi-algebraic sets has a long history. 
Given a semi-algebraic set described by a quantifier-free formula with polynomial atoms (for instance of the form $P>0$ with $P\in \D[X_1,\ldots,X_n]$), one can study decision and counting problems such as emptiness, the number of semi-algebraically connected components, higher Betti numbers, and the Euler--Poincar\'e characteristic; see, e.g., \cite{SS,BPRbettione,Bas05-first,BCL2019,BCT2020.1,BCT2020.2}. 
More recently, Basu and Karisani \cite{Basu-Karisani-2} developed algorithms for computing a different type of topological invariant---namely, persistent homology barcodes arising from filtrations of semi-algebraic sets by sublevel sets of continuous semi-algebraic functions. 
The present paper extends these one-parameter results to multi-parameter semi-algebraic filtrations and derives quantitative consequences, including complexity bounds and speed-type estimates.

\subsection{Complexity: singly vs.\ doubly exponential}
Complexity bounds for algorithms computing topological invariants are central to our results. 
We briefly summarize the current state of the art for computing invariants of semi-algebraic sets, distinguishing between doubly exponential and singly exponential regimes (with complexity measured as in Definition~\ref{def:complexity}).

\subsubsection{Doubly exponential}
Closed and bounded semi-algebraic sets $S\subset \R^n$ admit semi-algebraic triangulations, and given a quantifier-free description of $S$, such a triangulation can be computed effectively. 
However, the complexity of triangulation is doubly exponential in $n$: more precisely it is bounded by $(sd)^{2^{O(n)}}$, where $s$ is the number of polynomials appearing in the input and $d$ is their maximum degree.
Combining triangulation with standard linear-algebra routines yields an algorithm for computing all Betti numbers of $S$ with doubly exponential complexity.\footnote{Throughout the paper, all homology groups are taken with coefficients in a fixed field $\kk$.}
This approach is not sensitive to the homological degree: the doubly exponential dependence persists even if one seeks only low-dimensional invariants, e.g.\ the zeroth Betti number (the number of semi-algebraically connected components).

\subsubsection{Singly exponential}
In contrast, classical bounds from Morse theory yield singly exponential upper bounds on the Betti numbers of semi-algebraic sets \cite{OP,T,Milnor2,B99,GV07}.
On the algorithmic side, singly exponential algorithms are known for several fundamental problems, including:
computing the zeroth Betti number via roadmap methods \cite{Canny93a,GV92,BPR99,BRSS14,BR14}, and computing the Euler--Poincar\'e characteristic via Morse-theoretic techniques \cite{B99,BPR-euler-poincare}.
These algorithms run in time bounded by $(O(sd))^{n^{O(1)}}$.

This motivated the search for singly exponential algorithms for higher Betti numbers as well. 
The current state of the art is that, for each fixed $\ell\ge 0$, one can compute the first $\ell$ Betti numbers of a semi-algebraic set with complexity bounded by $(sd)^{n^{O(\ell)}}$ \cite{Bas05-first,BPRbettione,Basu-Karisani-1}.
By contrast, the best known complexity bound for computing \emph{all} Betti numbers of a semi-algebraic set in $\R^n$ remains doubly exponential in $n$ \cite{SS,Basu_survey}.

More recently, singly exponential algorithms have been obtained for several further tasks in semi-algebraic geometry, including: computing a semi-algebraic basis for the first homology group \cite{Basu-Percival}; computing homology (in fixed degrees) functorially on diagrams of semi-algebraic sets \cite{Basu-Karisani-3}; and, most relevant for the present work, computing low-dimensional barcodes of one-parameter semi-algebraic filtrations \cite{Basu-Karisani-2}. 
Beyond computational efficiency, developing singly exponential algorithms often reveals structural properties of semi-algebraic sets that are of independent interest (see Remark~\ref{rem:byproduct}).

\begin{remark}
\label{rem:condition}
Algorithms for computing Betti numbers have also been developed in numerical models of computation with explicit complexity estimates \cite{BCL2019,BCT2020.1,BCT2020.2}. 
In these results the complexity depends on a condition number associated to the input, and the algorithm may fail on ill-conditioned instances (when the condition number becomes infinite). 
The numerical setting is also methodologically distinct from our approach: we work in an algebraic model intended to apply uniformly over arbitrary real closed fields and ordered domains of coefficients (including non-Archimedean ones), with complexity bounds depending only on combinatorial input parameters such as $s$ and $d$, and not on the size of the coefficients. 
For these reasons the two kinds of results are not directly comparable.

\medskip
To the best of our knowledge, the numerical algorithms mentioned above have not been extended to persistent homology of semi-algebraic filtrations. Such an extension would require an appropriate notion of a condition number for a semi-algebraic filtration, which does not currently appear to be available. 
Finally, numerical algorithms typically do not yield quantitative structural consequences of the kind obtained here---for example, the speed bounds in Theorems~\ref{thm:speed} and \ref{thm:speed:uniform} (and other byproducts; see Remark~\ref{rem:byproduct}).
\end{remark}

\begin{remark}
\label{rem:byproduct}
A recurring theme in algorithmic semi-algebraic geometry is that a careful correctness and complexity analysis of an algorithm often yields quantitative mathematical statements of independent interest. 
For example, this perspective has led to quantitative versions of the curve selection lemma \cite{BR2021}, bounds on the radius of a ball guaranteed to intersect every semi-algebraically connected component of a given semi-algebraic set \cite{BR10}, and, more recently, effective bounds on the {\L}ojasiewicz exponent \cite{BN2024}, among many others.
In the same vein, the speed bounds proved in Theorems~\ref{thm:speed} and \ref{thm:speed:uniform} can be viewed as byproducts of our main algorithmic contribution---namely, the construction and complexity analysis of Algorithm~\ref{alg:main:ss}.
\end{remark}

\medskip
\subsection{Significance of Theorems~\ref{thm:main}, \ref{thm:main:mu} and \ref{thm:alg:main:mu}}
\label{subsec:significance:1-2}
Semi-algebraic filtrations arise naturally in a range of applications (see, for example, \cite{Basu-Karisani-2,Miller2015} and, more recently, \cite[Open Questions (ix), p.~433]{AMMW2024}). 
In such settings it is important to have explicit, quantitative procedures that (i) describe the induced multi-persistence module and (ii) reduce it effectively to finite restrictions that can be computed and compared.

\medskip
Our notion of complexity for the poset module associated to a semi-algebraic multi-filtration is intended to capture the intrinsic combinatorial/topological complexity of this module in a way that is compatible with algorithmic computation (as explained in the next paragraph).
The singly exponential bound proved in Theorem~\ref{thm:main} is in the same spirit as classical singly exponential bounds for topological invariants of semi-algebraic sets, such as Morse-theoretic bounds on Betti numbers \cite{OP,T,Milnor2,GV07}. 
Moreover, Theorem~\ref{thm:alg:main:mu} provides a singly exponential-time algorithm (for fixed $\ell$) to compute the barcode function $\mu_\ell(S,\mathbf f)$, extending the one-parameter result of \cite{Basu-Karisani-2} and paralleling singly exponential algorithms for low-dimensional Betti numbers in the non-persistent setting \cite{Bas05-first,BPRbettione,Basu-Karisani-1}.

It is worth emphasizing that the passage from $p=1$ to $p>1$ is not immediate: when $p>1$, the product order on $\R^p$ is only partial, and the resulting persistence invariants live on the much more complicated parameter space of comparable pairs in $\R^p\times \R^p$. 
Nevertheless, multi-filtrations are important in practice and arise in several applied and computational contexts \cite{CZ2007,CSZ2010,SCLRO2017}.

\subsection{Significance of Theorems~\ref{thm:speed} and \ref{thm:speed:uniform}}
\label{subsec:significance:3-4}

\subsubsection{Finite vs.\ infinite}
In the case of labelled directed graphs on the vertex set $[N]$ (allowing self-loops but no multiple edges), the number of distinct graphs is finite, namely $2^{N^2}$.

For finite persistence modules---that is, functors $\mathbf P:P\to \Vect_{\kk}$ from a finite poset $P$---the situation is more subtle. Fix a finite poset $P$ and an integer $M\ge 0$, and consider only those functors $\mathbf P:P\to \Vect_{\kk}$ satisfying the uniform dimension bound
\begin{equation}
\label{eqn:M}
\dim \mathbf P(a)\le M \qquad (a\in P).
\end{equation}
If $P$ is a chain (linearly ordered), then the isomorphism class of $\mathbf P$ is determined by the ranks of the structure maps $\mathbf P(a\preceq b)$, and since each rank lies in $\{0,1,\ldots,M\}$, there are only finitely many isomorphism classes of such functors.

In contrast, finiteness can fail dramatically for more general posets, even under the bound \eqref{eqn:M} with $M=2$.

\begin{example}
\label{eg:infinite}
Let $\kk$ be an infinite field, and let $P$ be the finite poset with Hasse diagram
\[
\xymatrix{
&& v && \\
w \ar[urr] & x\ar[ur] && y \ar[ul] & z \ar[ull]
}
\]
(i.e.\ $w,x,y,z\prec v$ and no other comparabilities). For $(a,b)\in \kk^2\setminus\{(0,0)\}$, define a poset module
$\mathbf P_{a,b}:P\to \Vect_{\kk}$ by
\[
\mathbf P_{a,b}(v)=\kk^2,
\qquad
\mathbf P_{a,b}(w)=\mathbf P_{a,b}(x)=\mathbf P_{a,b}(y)=\mathbf P_{a,b}(z)=\kk,
\]
and, for the covering relations into $v$,
\begin{align*}
\mathbf P_{a,b}(w\to v)&=[1,0]^t, &
\mathbf P_{a,b}(x\to v)&=[0,1]^t,\\
\mathbf P_{a,b}(y\to v)&=[1,1]^t, &
\mathbf P_{a,b}(z\to v)&=[a,b]^t,
\end{align*}
where the right-hand side denotes the matrix of the corresponding element of $\Hom_{\kk}(\kk,\kk^2)$ with respect to the standard bases.

\begin{claim}
If $[a:b]\neq [c:d]$ in $\PP^1_{\kk}$, then $\mathbf P_{a,b}\not\cong \mathbf P_{c,d}$. In particular, there are infinitely many pairwise non-isomorphic poset modules $\mathbf P:P\to \Vect_{\kk}$ satisfying $\dim \mathbf P(u)\le 2$ for all $u\in P$.
\end{claim}

\begin{proof}
Suppose $F:\mathbf P_{a,b}\to \mathbf P_{c,d}$ is an isomorphism of functors, and let $\tau=F(v)\in \mathrm{GL}_2(\kk)$.
Naturality forces $\tau$ to carry the images of the maps $w\to v$, $x\to v$, and $y\to v$ to themselves, i.e.\ it must preserve the three lines
\[
\kk(1,0)^t,\quad \kk(0,1)^t,\quad \kk(1,1)^t \subset \kk^2.
\]
The only invertible linear maps preserving these three distinct lines are scalar multiples of the identity; hence $\tau=\lambda I$ for some $\lambda\in \kk^\times$.
Therefore $\tau$ sends the line $\kk(a,b)^t$ to $\kk(c,d)^t$ if and only if these lines coincide, i.e.\ if and only if $[a:b]=[c:d]$ in $\PP^1_{\kk}$.
\end{proof}
\end{example}

Example~\ref{eg:infinite} shows that, even for a fixed finite poset and a uniform dimension bound $\le 2$, the set of strong (and hence also weak) equivalence classes of poset modules can be infinite.
This is precisely the phenomenon alluded to in Remark~\ref{rem:comparison}: even when the underlying finite poset is fixed and the vector-space dimensions are uniformly bounded, varying the linear maps can produce infinitely many non-isomorphic (and even non-weakly-equivalent) poset modules.
This highlights that the finiteness asserted by the speed bounds in Theorems~\ref{thm:speed} and \ref{thm:speed:uniform} is not automatic a priori (except in the one-parameter situation $p=1$, where the indexing posets are chains).

\medskip
Finally, unlike the case $p=1$, where one has a complete classification of pointwise finite-dimensional persistence modules via the structure theorem for finitely generated $\kk[X]$-modules (see, e.g., \cite{CZ2005}), the multi-parameter case $p>1$ is substantially more complicated and admits no comparable classification. There is therefore a large body of work devoted to computable invariants for multi-parameter persistence modules over finite posets \cite{miller2020homological,DX2022,DKM2024,BOO2022}. From this perspective, it is useful to emphasize that finite poset modules arising from semi-algebraic multi-filtrations form a very special subclass among all poset modules on the same underlying poset. In analogy with the rich theory of semi-algebraic graph classes \cite{Alon-Pach-et-al,Fox-et-al}, it would be interesting to investigate further structural properties of this subclass; we do not pursue this direction here.

\subsection{Prior work}
Persistent homology is a central object in topological data analysis \cite{dey2022computational,Weinberger_survey,Ghrist}, and it has also found applications across mathematics and scientific computing (see, for example, \cite{Ellis-King,Manin-Marcolli}). 
Persistent homology may be associated to any filtration of spaces; it refines ordinary homology, which is recovered from the constant filtration. 
Much of the classical theory and many computational tools were developed first in the one-parameter setting.

\medskip
In recent years, multi-parameter persistence has become an active area of research. 
Because persistence modules indexed by $\R^p$ (with the product order) do not admit a classification analogous to the one-parameter barcode theorem, a substantial literature has focused on structural frameworks and computable invariants for multi-parameter modules; see, for instance, \cite{Miller2023,arya2023} and the references therein.

\medskip
From a different viewpoint, the study of persistence for semi-algebraic filtrations using methods from algorithmic real algebraic geometry is more recent. 
In the one-parameter semi-algebraic setting, Basu and Karisani \cite{Basu-Karisani-2} introduced an effective barcode formalism and provided singly exponential algorithms (in fixed homological degree) to compute it. 
On the abstract side, Kashiwara--Schapira \cite{KS2018} (and related work such as \cite{Miller2023,Berkouk2023,Curry-Patel}) recast persistence in the language of constructible sheaves, where persistence-type invariants arise functorially via sheaf-theoretic direct images and the closure properties afforded by the six-functor formalism (see, e.g., \cite{Kashiwara-book}). 
These sheaf-theoretic approaches provide a powerful conceptual framework and apply in settings beyond the semi-algebraic category (e.g.\ subanalytic filtrations), but they are not designed to yield the kind of explicit, quantitative complexity bounds that are the focus of the present paper.
Our results can be viewed as providing, within the semi-algebraic category, a constructive and quantitative counterpart to this general sheaf-theoretic perspective.

\section{Proofs of Theorems~\ref{thm:main}, \ref{thm:main:mu}, 
and \ref{thm:alg:main:mu}
}
\label{sec:proof:1-2}

\subsection{Effective construction of semi-algebraic poset modules}
\label{subsec:input-output}

We first describe an algorithm which, given a semi-algebraic multi-filtration $(S,\mathbf f)$ and an integer $\ell\ge 0$, produces explicit semi-algebraic data encoding the poset modules
$\mathbf P_{S,\mathbf f,\ell}$ and its normalized restriction $\mathbf P_{S,\mathbf f,\ell}\!\restriction_{\PP_p}$.
The correctness of the construction implies that these modules are semi-algebraically constructible, and the accompanying complexity analysis yields explicit upper bounds on their complexities in the sense of Definition~\ref{def:constructible-poset-module}. This will be the main input in the proof of Theorem~\ref{thm:main}.

\medskip
More precisely, we specify an algorithm with the following input--output behavior.

\noindent\textbf{Input.}
\begin{enumerate}[(a)]
\item finite sets of polynomials
$\mathcal P \subset \D[X_1,\ldots,X_n]$ and
$\mathcal Q \subset \D[X_1,\ldots,X_n,Y_1,\ldots,Y_p]$;
\item a $\mathcal P$-closed formula $\Phi$ with $\RR(\Phi)=S\subset \R^n$;
\item a $\mathcal Q$-closed formula $\Psi$ with $\RR(\Psi)=\mathrm{graph}(\mathbf f)\subset \R^{n+p}$, where $\mathbf f:S\to \R^p$ is continuous and $\mathcal Q$-closed;
\item an integer $\ell\ge 0$.
\end{enumerate}

\noindent\textbf{Output.}
\begin{enumerate}[(a)]
\item an integer $N>0$;
\item a finite set $\mathcal C\subset \D[Y_1,\ldots,Y_p]$ containing the polynomials $Y_i\pm 1$ ($1\le i\le p$), and for each
$C\in \Cc(\mathcal C)$ an integer $N_C\ge 0$;
\item a finite set $\mathcal D\subset \D[Y_1,\ldots,Y_p,Y_1',\ldots,Y_p']$ containing, for each $1\le i\le p$, the polynomials
\[
\ Y_i\pm 1,\ Y_i'\pm 1,\ Y_i-Y_i',
\]
and for each cell $D\in \Cc(\mathcal D)$ with $D\subset \Pairs(\R^p)$, a matrix $M_D\in \kk^{N\times N}$.
\end{enumerate}
(Thus $\PP_p$ and $\PP_p^{\max}$ are $\mathcal C$-semi-algebraic, and the sets $\Pairs(\R^p)$, $\PP_p^o\times \PP_p$, and $\PP_p^o\times \PP_p^{\max}$ are $\mathcal D$-semi-algebraic.)

\medskip
The output satisfies the following property.

\begin{property}
\label{property:input-output}
The poset module $\mathbf P_{S,\mathbf f,\ell}$ (resp.\ $\mathbf P_{S,\mathbf f,\ell}\!\restriction_{\PP_p}$) is isomorphic (Definition~\ref{def:poset-module-isomorphism}) to the semi-algebraically constructible poset module
\[
\mathbf P:(\R^p,\preceq)\to \Vect_{\kk}
\qquad
\text{(resp.\ $\mathbf P:(\PP_p,\preceq)\to \Vect_{\kk}$)}
\]
defined as follows (see Notation~\ref{not:matrix}).
For $\y\in \R^p$, let $C(\mathcal C,\y)$ denote the unique cell $C\in \Cc(\mathcal C)$ containing $\y$, and for $(\y,\y')\in \R^p\times \R^p$ let
$C(\mathcal D,(\y,\y'))$ denote the unique cell $D\in \Cc(\mathcal D)$ containing $(\y,\y')$.
Then
\begin{align*}
\mathbf P(\y) &= \kk^{\,N_{C(\mathcal C,\y)}},\\
\mathbf P(\y\preceq \y') &= 
L_{\Trunc^{N}_{\,N_{C(\mathcal C,\y')},\,N_{C(\mathcal C,\y)}}\!\bigl(M_{C(\mathcal D,(\y,\y'))}\bigr)}.
\end{align*}
\end{property}

\subsection{Algorithmic preliminaries}

\subsubsection{Algorithm for enumerating $\Cc(\mathcal{P})$}
\label{subsec:enumerating-CC}
One basic algorithm with singly exponential complexity is an algorithm
that given a finite set $\mathcal{P} \subset \D[X_1,\ldots,X_n]$ as input computes a finite set of points guaranteed to intersect every $C \in \Cc(\mathcal{P})$  (see for example \cite[Algorithm 13.2 (Sampling)]{BPRbook2}). This algorithm in conjunction with an algorithm 
(see for example \cite[Algorithm 16.4 (Uniform Roadmap)]{BPRbook2})
for computing roadmaps of semi-algebraic sets, gives an algorithm for
computing exactly one point in every $C \in \Cc(\mathcal{P})$ -- and thus also enumerate the elements of $\Cc(\mathcal{P})$. Taking into account the complexity of \cite[Algorithm 13.2 (Sampling)]{BPRbook2}) and
\cite[Algorithm 16.4 (Uniform Roadmap)]{BPRbook2}), the complexity of the resulting algorithm is bounded by $(sd)^{n^{O(1)}}$, where
$s = \card(\mathcal{P})$ and $d = \deg(\mathcal{P})$.
We are going to use this algorithm for enumerating the elements of 
$\Cc(\mathcal{P})$ and computing exactly one point in each $C \in \Cc(\mathcal{P})$ implicitly without mentioning in the description of the more complicated algorithms that we describe later.

\subsubsection{Parametrized versions of algorithms in real algebraic geometry}
\label{subsec:parametrized}

Many algorithms in real algebraic geometry (see for example \cite[Chapters 13-16]{BPRbook2}) have the following form. They take as 
input a finite set $\mathcal{P}$ of polynomials with coefficients in an ordered  domain $\D$ contained in a real closed field $\R$, and also a $\mathcal{P}$-formula,
and produces as output 
discrete objects (for example, a list of sign conditions on a set of polynomials, or a simplicial complex etc. or some topological invariant like the Betti numbers of the $\mathcal{P}$-semi-algebraic set specified in the input etc.), as well as some algebraically defined objects, such as set of points represented by real univariate representations, or 
a set of semi-algebraic curves represented by parametrized real univariate representations, or even more generally semi-algebraic maps 
whose graphs are described by formulas. We will refer to these parts as discrete and the algebraic parts of the output respectively.

\medskip
We will need to use \emph{parametrized versions} of three different algorithms.
By parametrized version of a given algorithm $\mathbf{A}$  with parameters $Y = (Y_1,\ldots,Y_p)$ we mean an algorithm that takes as input a finite set $\mathcal{P}$ of polynomials with coefficients in $\D[Y_1,\ldots,Y_p]$ (instead of $\D$). 
The output of the algorithm consists of two parts.
\begin{enumerate}[(a)]
    \item A finite set $\mathcal{F} \subset \D[Y_1,\ldots,Y_p]$, having the property that for each $C \in \Cc(\mathcal{F})$, the discrete part of the output of the Algorithm $\mathbf{A}$, with input $\mathcal{P}(\y,\cdot)$, is the same for all $\y \in C$;
    
    \item For each $C \in \Cc(\mathcal{F})$:
        \begin{enumerate}
            \item the discrete part of the output of Algorithm $\mathbf{A}$ with input $\mathcal{P}(\y,\cdot)$ some (or all) $\y \in C$ (so this is constant as 
            $\y$ varies over $C$);
            \item the  (varying) algebraic part of the output of Algorithm $\mathbf{A}$ parametrized by $Y$, given by formula(s), $\psi_C(Y)$ so that for all $\y \in C$, the output of Algorithm $\mathbf{A}$ with input $\mathcal{P}(\y,\cdot)$ is obtained by specializing $Y$ to $\y$ i.e. is equal to $\psi_C(\y)$.    
        \end{enumerate}
\end{enumerate}

\begin{example}[Basic example 1: $\Elim_X$]
    \label{eg:Elim}
    Let $\mathcal{P} \subset \D[Y_1,\ldots,Y_p,X]_{\leq d}$ and 
    \[
    \phi(Y_1,\ldots,Y_p,X)
    \] be a $\mathcal{P}$-closed formula 
    so that
    $\RR(\phi(\y,X)$ is bounded for all $\y \in \R^p$. We consider the problem of parametrized triangulation of $\RR(\phi(\y,X)) \subset \R$ as $\y$ varies over $Y$. Let  
    $\mathcal{F} = \Elim_X(\mathcal{P}) \subset \D[Y_1,\ldots,Y_p]$
    where $\Elim_X(\mathcal{P})$ is defined in \cite[Notation 5.15]{BPRbook2}.
    The for each $C \in \Cc(\mathcal{F})$, there exists continuous semi-algebraic functions $\xi_1,\ldots,\xi_{M_C}: C \rightarrow \R$, with
    $\xi_1 < \ldots < \xi_{M_C}$, and such that for each $\y \in C$,
    $\xi_1(\y)< \cdots < \xi_{M_C}(\y)$ are the ordered set of real roots of all the polynomials in $\mathcal{P}(\y,X)$ which are not identically $0$ on $C$ \cite[Theorem 5.16]{BPRbook2}. Thus, one obtains a triangulation 
    (i.e. closed intervals and their end-points) of $\RR(\phi(\y,X))$
    which has the same combinatorial structure for all $\y \in C$.

\medskip
    It follows from the definition of $\mathcal{F}$ that
    $\card(\mathcal{F}) \leq s d^{O(1)}$, and $\deg(\mathcal{F}) \leq d^{O(1)}$. Finally, the complexity of the parametrized algorithm 
    for triangulating a semi-algebraic subset of $\R$, is bounded by 
    $(s d)^{O(p)}$.
\end{example}
\subsubsection{Parametrized algorithm for computing semi-algebraic triangulation}
The technique described in Example~\ref{eg:Elim} can be extended 
iteratively (see \cite[Chapter 5]{BPRbook2}) to obtain an algorithm for triangulating closed and bounded semi-algebraic sets in higher dimensions.
More precisely, we will use a parametrized version of the following algorithm.

\begin{algorithm}[H]
\caption{(Triangulation of semi-algebraic sets)}
\label{alg:triangulation}
\begin{algorithmic}[1]
\INPUT
\Statex{
\begin{enumerate}[(a)]
\item 
a finite set of $s$ polynomials ${\mathcal P} \subset \D[X_1,\ldots,X_n]$;
\item ${\mathcal P}$-formulas $\phi_1,\ldots,\phi_N$ such that
for each $i, 1 \leq i \leq N$, $\RR(\phi)$ is a closed and bounded semi-algebraic set. 
\end{enumerate}
}
\OUTPUT
\Statex{
\begin{enumerate}[(a)]
\item A finite set of polynomials $\mathcal{Q} \subset \R[X_1,\ldots,X_n]$;
\item 
    a finite simplicial complex
    $K$, and subcomplexes $K_1,\ldots,K_N \subset K$;
\item a $\mathcal{Q}$-formula, $\psi$, such that $\RR(\psi)$ is the graph of a semi-algebraic homeomorphism $h:|K| \rightarrow \bigcup_{i=1}^N S_i$,
    which restricts to a semi-algebraic homeomorphisms 
    $h_i = h|_{|K|}: |K_i| \rightarrow \RR(\phi_i)$.
    \end{enumerate}
}
\COMPLEXITY
The complexity of the algorithm is bounded by  
$N (s d)^{2^{O(n)}}$, and the size of the simplicial complex $\Delta$ is also bounded by $(s d)^{2^{O(n)}}$.
\end{algorithmic}
\end{algorithm}

The parametrized version
of Algorithm~\ref{alg:triangulation} will
output a finite subset 
\[
\mathcal{F} \subset \D[Y_1,\ldots,Y_p],
\]
with 
\[
\card(\mathcal{F}) \leq (sd)^{2^{O(n)}}, \deg(\mathcal{F}) \leq d^{2^{O(n)}},
\]
as well as
a finite set of polynomials $\mathcal{Q} \subset \R[Y_1,\ldots,Y_p, X_1,\ldots,X_n]$, and for each $C \in \Cc(\mathcal{F})$,
\begin{enumerate}[(a)]
\item 
    a finite simplicial complex
    $K_C$, and subcomplexes $K_{C,1},\ldots,K_{C,N} \subset K_C$;
\item a $\mathcal{Q}$-formula, $\psi_C$, such that $\RR(\psi_C(\y,X)$ is the graph of a semi-algebraic homeomorphism $h_C:|K_C| \rightarrow \bigcup_{i=1}^N \RR(\phi_i(\y,X))$,
    which restricts to a semi-algebraic homeomorphisms 
    $h_{C,i} = h|_{|K_C|}: |K_{C,i}| \rightarrow \RR(\phi_i(\y,X))$.
    \end{enumerate}
The complexity of this parametrized algorithm is bounded by 
\begin{equation}
\label{eqn:complexity:triangulation}
(sd)^{2^{O(n)} p^{O(1)}}.
\end{equation}

The parametrized version (with parameters $Y = (Y_1,\ldots,Y_p)$) is obtained by following the same algorithm as the unparametrized case
with $n+p$ variables, applying the $\Elim_{X_i}$ operator successively
for $i=n,n-1,\ldots,1$ and taking $\mathcal{F} \subset \D[Y_1,\ldots,Y_p]$ the resulting family of polynomials. In order to obtain a triangulation one needs to make also a linear change in the
$X$-coordinates which we are overlooking here (\cite[Chapter 5]{BPRbook2} for details). The complexity of the algorithm is obtained by tracking the degree and the cardinality of the sets of polynomials obtained after applying the $\Elim_{X_i}$ operators, and these square at each step leading to a doubly exponential (in $n$) complexity.

\begin{example}[Basic Example 2: $\BElim_X$]
    \label{eg:BElim}
    The method of applying the $\Elim_X$ operator explained in Example~\ref{eg:Elim} one variable at a time leads to doubly exponential complexity. There is a more efficient algorithm which is the basis of all singly exponential algorithm, which is based on the  ``critical point method'' that is used to eliminate a block of variables at one time. 

    Given $\mathcal{P} \subset \D[Y_1,\ldots,Y_p,X_1,\ldots,X_n]_{\leq d}$,
    there exists a finite set 
    \[
    \mathcal{F} := \BElim_X(\mathcal{P}) \subset \D[Y_1,\ldots,Y_p]
    \]
    (see \cite[Algorithm 14.1]{BPRbook2} for a precise definition)
    such that for each $C \in \Cc(\mathcal{F})$, the set 
    $\Sign(\mathcal{P}(\y,X))$ stay invariant as $\y$ varies of $C$
    \cite{BPRbook2}. Moreover, $\card(\mathcal{F}) \leq s^{n+1} d^{O(n)}$, $\deg(\mathcal{F}) \leq d^{O(n)}$. 
    While the property ensured by $\BElim_X$ is weaker than that
    in the case of repeated application of $\Elim_X$, it is still the basis of (the parametrized versions) of several singly exponential complexity algorithms that we will use later (Algorithms~\ref{alg:cover} and \ref{alg:poset}).
\end{example}

\subsection{Topological preliminaries}
\label{subsec:top}
We also need the following basic facts from algebraic topology.

\begin{definition}[Closed cover associated to a triangulation]
\label{def:triangulation-cover}
    Given a semi-algebraic triangulation $h: |\Delta| \rightarrow S$, where $S$ is a closed and bounded semi-algebraic set, we denote by
    $\Cov(h)$ the covering of $S$ by the images of the closed simplices of $|\Delta|$.
\end{definition}

\begin{definition}[Nerve of a closed cover]
\label{def:nerve-cover}
    Given a cover $C$ of a closed and bounded semi-algebraic set $S$,
    we denote by $\Nerve(C)$, the nerve complex of $C$. We say that $C$ is a good cover if each element of $C$ is semi-algebraically contractible. In particular, if $h: |\Delta| \rightarrow S$ is a semi-algebraic triangulation then, $\Cov(h)$ is a good cover. 
\end{definition}

The following proposition is classical (see for example \cite{Bjorner}).
\begin{proposition}[Nerve Lemma]
\label{prop:nerve}
    If $C$ is a good cover of $S$, then there is a  canonical isomorphism
    $\psi_C: \HH_*(\Nerve(C)) \rightarrow \HH_*(S)$.
\end{proposition}
One immediate corollary is the following.

\begin{corollary}
    There is a canonical isomorphism 
    \[
    \psi_{\Cov(h)}: \HH_*(\Nerve(\Cov(h))) \rightarrow \HH_*(S).
    \]
\end{corollary}

\begin{proof}
    Follows from Definitions~\ref{def:triangulation-cover}and \ref{def:nerve-cover}, and Proposition~
    \ref{prop:nerve}.
\end{proof}

\begin{proposition}
\label{prop:refinement}
   If $C' \subset C$ are two good covers of a closed and bounded semi-algebraic set $S$, then $\Nerve(C') \subset \Nerve(C)$, and the induced
   map $\HH_*(\Nerve(C')) \rightarrow \HH_*(\Nerve(C))$ is an isomorphism.
\end{proposition}

\begin{proof}
    We have the following commutative diagram 
    \[
    \xymatrix{
\HH_*(\Nerve(C')) \ar[rr]\ar[rd]^{\psi_{C'}} && \HH_*(\Nerve(C)) \ar[ld]^{\psi_C} \\
    & \HH_*(S)&
    }
  \]
where the $\psi_C',\psi_C$ are isomorphisms using Proposition~\ref{prop:nerve}.
This implies that the horizontal arrow is an isomorphism as well.
\end{proof}

\begin{definition}[Refinement of a semi-algebraic triangulation]
Given two semi-algebraic triangulations $h_1: |\Delta_1| \rightarrow S$,
$h_2: |\Delta_2| \rightarrow S$, we say that $h_2$ is a refinement of 
$h_1$, if the image by $h_2$ of every closed simplex of $|\Delta_2|$ is contained in the image by $h_1$ of some some closed simplex of $\Delta_1$.
\end{definition}

\begin{proposition}
\label{prop:refinement2}
    Suppose $h_1,h_2$ are two triangulations of a closed and bounded semi-algebraic set $S$ such that $h_2$ is a refinement of $h_1$. Then, $\Cov(h_1) \cup \Cov(h_2)$ is s good cover of $S$, and the following diagram, where the vertical arrow is induced by inclusion is commutative and all arrows are isomorphisms:
    \[
    \xymatrix{
    \Nerve(\Cov(h_1) \cup \Cov(h_2)) \ar[rrrr]^{\;\;\Psi_{\Cov(h_1) \cup \Cov(h_2)}} &&&& \HH_*(S) \\
    \Nerve(\Cov(h_1)) \ar[u] \ar[rrrru]_{\Psi_{\Cov(h_1)}} &&&&
    }
    \]
\end{proposition}

\begin{proof}
    Follows from Propositions~\ref{prop:nerve} and \ref{prop:refinement}.
\end{proof}

\subsection{Algorithm for constructing $\mathbf{P}_{S,\mathbf{f},\ell}$ and 
$\mathbf{P}_{S,\mathbf{f},\ell}\!\restriction_{\PP_p}$}
In order to explain the idea behind 
the algorithm better,
we first
describe an algorithm 
whose complexity is bounded
by $(sd)^{2^{O(n)}p^{O(1)}}$ which is doubly exponential -- rather than
$(sd)^{( p n)^{O(\ell)}}$.
The doubly exponential complexity arises from using a semi-algebraic triangulation 
algorithm which has inherently doubly exponential complexity (but is easier to understand). Later we will replace the triangulation by a weaker construction -- namely, parametrized algorithm for simplicial replacement (Algorithm~\ref{alg:poset}) which is sufficient for our purposes and 
has the right complexity, but more difficult to visualize.

\subsubsection{Algorithm for semi-algebraic multi-persistence using effective semi-algebraic triangulation}
We outline the main idea first. While reading the outline 
it will be helpful if the reader also consults the description of the algorithm that follows (Algorithm~\ref{alg:main:triangulation}) 
for the technical details of the various steps.  

\medskip
We begin by obtaining a semi-algebraic partition of $\R^p$, and over each element (say $C$) of the partition, a simplicial complex $K_C$, such that for all $\y \in C$, $S_\y$ is semi-algebraically homeomorphic to $|K_C|$ (we denote by $|\cdot|$ the geometric realization of a simplicial complex as a polyhedron in 
some $\R^N$). This first step uses parametrized version of Algorithm~\ref{alg:triangulation}(Semi-algebraic triangulation) and is implemented in \textbf{Lines~\ref{alg:main:triangulation:1} - \ref{alg:main:triangulation:4}} in Algorithm~\ref{alg:main:triangulation}.

\medskip
We work  with  the $(\ell+1)$-skeleton of the  
nerve complex of the induced closed covering -- and we denote this simplicial complex by $\Delta_C$ (\textbf{Line~\ref{alg:main:triangulation:5}}). There is thus (using Proposition~\ref{prop:nerve}) a canonical isomorphism $\widetilde{\psi}_{C,\y,i}:\HH_i(S_\y) \rightarrow \HH_i(\Delta_C)$ for all $\y \in C $ and $0 \leq i \leq \ell$.

\medskip
For $(\y,\y') \in C \times C' \cap 
(\PP_p^o) \times \PP_p) \cap \Pairs(\PP_p)
$, 
there \emph{does not exist} any canonically defined
homomorphism $\widetilde{\phi}_{C,C',i}: \HH_i(\Delta_{C}) \rightarrow \HH_i(\Delta_{C'})$ such that the following diagram
commutes for every $(\y,\y') \in (C \times C') \cap \ 
(\PP_p^o \times \PP_p) \  \cap \Pairs(\PP_p)
$:
\[
\xymatrix{
\HH_i(S_\y) \ar[rr]^{\iota^{\y,\y'}_i}\ar[d]^{\widetilde{\psi}_{\y,i}}&&
\HH_i(S_{\y'}) \ar[d]^{\widetilde{\psi}_{\y',i}} \\
\HH_i(\Delta_C) \ar[rr]^{\widetilde{\phi}_{C,C',i}}
&& \HH_i(\Delta_{C'})
}
\]

In order to obtain homomorphisms between $\HH_*(C)$ and 
$\HH_*(\Delta_{C'})$ making the above diagram to commute, we further
partition semi-algebraically $C \times C' \ \cap 
(\PP_p^o \times \PP_p) \ \cap \  \Pairs(\PP_p)
$. Our tool is again the parametrized version of Algorithm~\ref{alg:triangulation}. Using it we obtain parametrically 
with $(\y,\y') \in C \times C' \cap \ 
(\PP_p^o \times \PP_p) \  \cap \  \Pairs(\PP_p)
$ as parameter, refinement of the triangulations $K_C,K_{C'}$. This leads to a further partition of $C \times C' \ \cap \ 
(\PP_p^o \times \PP_p) \ \cap \ \Pairs(\PP_p)
$, so that in each part $D$ of this partition, and 
$(\y,\y') \in D$, we have a triangulation of $S_{\y'}$ which is a refinement of the triangulation of $C'$ obtained in the previous step, and also its restriction to $C$ is a refinement of the triangulation obtained in the previous step. Denoting the $(\ell+1)$-skeleton of the nerve complex of the cover induced by this new triangulation $\Delta_D$, we have that both $\Delta_{C}$ and 
$\Delta_{C'}$ are in a natural way sub-complexes of $\Delta_D$,
and there are canonical isomorphisms (using Proposition~\ref{prop:refinement2}),
$\HH_i(S_\y) \rightarrow \HH(\Delta_D^0)$,
$\HH_i(S_{\y'}) \rightarrow \HH(\Delta_D)$ (where $\Delta_D^0 \subset \Delta_D$ is the subcomplex corresponding to the inclusion
$S_\y \subset S_{\y'}$).
This step is implemented in \textbf{Line~\ref{alg:main:triangulation:6}} of
Algorithm~\ref{alg:main:triangulation}.

\medskip
Notice that we also have canonically defined  \emph{isomorphisms}, 
$\HH_i(\Delta_C) \rightarrow \HH_i(\Delta_D^0)$,
$\HH_i(\Delta_{C'}) \rightarrow \HH_i(\Delta_D)$,
induced by the inclusion of a subcomplex in another,
and also a canonically defined homomorphism,
$\HH_i(\Delta_D^0) \rightarrow \HH_i(\Delta_D)$ (also induced by inclusion) that depends on $D$ but is independent of 
the choice of $(\y,\y') \in D$. In short we have, for all $(\y,\y') \in D$, a commutative diagram where all solid arrows are canonically defined:

\begin{equation}
\label{eqn:commutative}
\xymatrix{
\HH_i(S_\y) \ar[rr]^{\iota^{\y,\y'}_i}\ar[d]^{\cong} \ar@/_2pc/[dd]_{\widetilde{\psi}_{\y,i}}&& \HH_i(S_{\y'})\ar[d]_{\cong} \ar@/^2pc/[dd]^{\widetilde{\psi}_{\y',i}}\\
\HH_i(\Delta_D^0) \ar[rr]\ar[d]^{\cong} && \HH_i(\Delta_D) \ar[d]_{\cong} \\
\HH_i(\Delta_C) \ar@{.>}[rr]^{\widetilde{\phi}_{C,C',D,i}}&&\HH_i(\Delta_{C'})
}
\end{equation}

The dotted arrow is then uniquely defined as the one which makes the whole diagram commute -- and we define it to be the homomorphism $\widetilde{\phi}_{C,C',D,i}$
(Line~\ref{alg:main:triangulation:4}).
Notice that for all $\y \preceq \y' \preceq \y''$, 
with $\y \in C, \y' \in C', \y'' \in C''$,
and $(\y,\y') \in D, (\y',\y'') \in D', (\y',\y'') \in D''$
we have the commutative diagram: 
\\\\

\[
\xymatrix{
\HH_i(S_\y) \ar[rr]^{\iota^{\y,\y'}_i}\ar[d]^{\widetilde{\psi}_{\y}} \ar@/^2pc/[rrrr]^{\iota^{\y,\y''}_i} && \HH_i(S_{\y'}) \ar[rr]^{\iota_{\y',\y'',i}}\ar[d]^{\widetilde{\psi}_{\y'}} && \HH_i(S_{\y''}) \ar[d]^{\widetilde{\psi}_{\y''}} \\
\HH_i(\Delta_C) \ar[rr]^{\widetilde{\phi}_{C,C',D,i}} \ar@/_2pc/[rrrr]^{\widetilde{\phi}_{C,C'',D'',i}}  && \HH_i(\Delta_{C'}) \ar[rr]^{\widetilde{\phi}_{C',C'',D',i}}  && \HH_i(\Delta_{C''}) 
}
\]

\medskip
These are the key steps. The rest of the steps of the algorithm 
(\textbf{Lines~\ref{alg:main:triangulation:7} - \ref{alg:main:triangulation:12}})
consists of using standard algorithms from linear algebra to compute
bases of the various  $\HH_i(\Delta_C)$, $\HH_i(\Delta_{C'})$ 
$\HH_i(\Delta_{D}^0)$, $\HH_i(\Delta_{D})$,
and
the matrices with respect to these bases of the maps between these spaces induced by refinements and inclusions
(the solid arrows in the bottom square of the commutative diagram ~\eqref{eqn:commutative}), 
and finally the matrices of the maps
$\widetilde{\phi}_{C,C',D,i}$ (shown by the dotted arrow in 
\eqref{eqn:commutative}).

The complexity of the above algorithm is dominated by the complexity of the calls to the parametrized version of the semi-algebraic triangulation algorithm. Each such call costs $(s d)^{2^{O(n)}p^{O(1)}}$ -- and which is also asymptotically an upper bound on the total complexity.
Notice that the upper bound on the complexity is doubly exponential in $n$ and independent of $\ell$. 

\medskip
We now describe more formally the algorithm outlined above.

\begin{algorithm}[H]
\caption{(Semi-algebraic multi-persistence using triangulations)}
\label{alg:main:triangulation}
\begin{algorithmic}[1]
\INPUT
\Statex{
As stated in 
Section~\ref{subsec:input-output}.
}
\OUTPUT
\Statex{
As stated in 
Section~\ref{subsec:input-output}.
}

\PROCEDURE

\State{
Denote
\begin{eqnarray*}
\widetilde{S} &=& \{(\x,\y) \in \R^n \times \R^p \mid \x \in S, \mathbf{f}(\x) \preceq \y\}.    
\end{eqnarray*}
}
\label{alg:main:triangulation:1} 

\State{
Let $\widetilde{S} = \RR(\widetilde{\phi})$, where $\widetilde{\phi}$
is a closed formula which is obtained from the closed formulas
describing $S$ and $\mathrm{graph}(\mathbf{f})$ in a straightforward way.
}
\label{alg:main:triangulation:2}

\State{ 
Using the parametrized version of Algorithm~\ref{alg:triangulation}, with $Y = (Y_1,\ldots,Y_p)$ as parameters and the formula 
$\RR(\widetilde{\phi})$ as input to obtain:
\begin{enumerate}[(a)]
    \item a finite set of polynomials 
$\mathcal{C}  \subset \D[Y_1,\ldots,Y_p]$;
    \item for each $C \in \Cc(\mathcal{C})$:
        \begin{enumerate}[(i)]
            \item a simplicial complex $K_C$;
            \item a formula $\Phi_C(Y,\cdot)$, such that for 
                    each $\y \in C$, $\RR(\Phi(\y,\cdot))$ equals the graph 
                    of a semi-algebraic homeomorphism 
                    \[
                        h_{C,\y}: |K_C| \rightarrow \widetilde{S}_\y.
                    \]
        \end{enumerate}
\end{enumerate}
}
\label{alg:main:triangulation:3}

\algstore{myalg}
\end{algorithmic}
\end{algorithm}
 
\begin{algorithm}[H]
\begin{algorithmic}[1]
\algrestore{myalg}

\State{
For each simplex $\tau$ of $K_C$,  denote by          
$\Psi_{C,\tau}(Y,X)$,
the formula obtained in an obvious way from $\Phi_C$, such that for each $\y \in C$,          
$\RR(\Psi_{C,\tau}(\y,X)) = h_{C,\y}(|\tau|)$.
}

\label{alg:main:triangulation:4}

\State{
                    For each $\y \in C, C \in \Cc(\mathcal{C})$, denote
                    \[
                        \Delta_C = \sk_{\ell+1}(\Nerve(\Cov(h_C,\y))),
                    \]
                    noting that this definition
                    depends only on $C$ and independent of the choice of $\y \in C$. 
}
\Comment{Note that, for $\y \in C$, and $0 \leq i \leq \ell$, there is a 
isomorphism (depending only on $h_{C,\y}$)
\[
\widetilde{\psi}_{C,\y,i}: \HH_i(\widetilde{S}_\y) \rightarrow \HH_i(\Delta_C).
\]
}
\label{alg:main:triangulation:5}

\State{ 
For each pair $C,C' \in \Cc(\mathcal{C})$, such that 
\[
C \times C' \cap 
(\PP_p^o \times \PP_p) \cap \Pairs(\PP_p)
\neq \emptyset,
\]
using parametrized version of Algorithm~\ref{alg:triangulation} again, with parameters 
\[
(Y,Y') = (Y_1,\ldots,Y_p,Y_1',\ldots,Y_p'),
\]
and input 
\[
(\Psi_{C,\tau})_{\tau \in K_C} \cup (\Psi_{C',\tau})_{\tau \in K_{C'}},
\]
obtain:

    \begin{enumerate}[(a)]
    \item
        a finite set of polynomials $\mathcal{D}_{C,C'} \subset \D[Y_1,\ldots,Y_p,Y_1',\ldots,Y_p']$,
        (we will also assume without loss of generality, by including them if necessary, that the polynomials $Y_ \pm 1, Y_i'\pm 1, Y_i' - Y_i, 1\leq i \leq p$ 
        belong to $\mathcal{D}_{C,C'}$);

    \item
        for each $E \in \Cc(\mathcal{D}_{C,C'})$ such that 
        \[
            C \times C' \cap 
            (\PP_p^o \times \PP_p) \cap \Pairs(\PP_p)
            \cap E \neq \emptyset,
        \]
        a formula 
        \[
        \Phi_{C,C',E}(Y,Y',\cdot),
        \]
        such that for each $(\y,\y')  \in C\times C' \cap 
        (\PP_p^o \times \PP_p) \cap \Pairs(\PP_p)
        \cap E$,
        \[
        \RR(\Phi_{C,C',E}(\y,\y',\cdot))
        \] 
        equals the graph of a semi-algebraic triangulation 
        \[
            h_{C,C',E,\y,\y'}: |K_{C,C',E}| \rightarrow \widetilde{S}_{\y'}.
        \]
       
        Moreover, $h_{C,C',E,\y,\y'}$ refines the semi-algebraic triangulation $h_{C,\y'}$, 
        and the restriction, $h_{C,C',E,\y,\y'}^0$,
        of $h_{C,C',E,\y,\y'}$ to $h_{C,C',E,\y,\y'}^{-1}(\widetilde{S}_{\y})$,
        is a refinement of the triangulation $h_{C,\y}$.
        \end{enumerate}
        
        \Comment{Note that the simplicial complexes $\Nerve(\Cov(h_{C,C',
        E,\y,\y'}^0)) \subset 
        \Nerve(\Cov(h_{C,C',E,\y,\y'}))$ are independent of the choice of $\y,\y'$.)}
        Denote 
        \[
        \Delta_{C,C',E}^0 = \sk_{\ell+1}(\Nerve(\Cov(h_{C,C',E,\y,\y'}^0))),
        \]
        \[
        \Delta_{C,C',E}^0 = \sk_{\ell+1}(\Nerve(\Cov(h_{C,C',E,\y,\y'}^0))),
        \]
        noting that
        $\Delta_{C,C',D}^0 \subset \Delta_{C,C',D}$.

}
\label{alg:main:triangulation:6}

\State{
Denote 
    \[
    \mathcal{D} = \mathcal{C}(Y) \cup \mathcal{C}(Y') \cup \bigcup_{(C,C')} \mathcal{D}_{C,C'},
    \] 
    where the last union is taken over
    all $(C,C')$ such that $C \times C \cap 
    (\PP_p^o \times \PP_p) \cap \Pairs(\PP_p)
    \neq \emptyset $.

    Note that for each $D \in \Cc(\mathcal{D})$, there
    exists, uniquely defined $C_1(D), \C_2(D) \in \Cc(\mathcal{C})$, and
    $E(D) \in \Cc(\mathcal{D}_{\sigma,\sigma'})$
    such that $D\subset \C_1(D) \times C_2(D) \cap E(D
    )$.
 }
\label{alg:main:triangulation:7}

\algstore{myalg}
\end{algorithmic}
\end{algorithm}
 
\begin{algorithm}[H]
\begin{algorithmic}[1]
\algrestore{myalg}

 \State{
 For each $D \in \Cc(\mathcal{D})$, such that $D \subset 
 (\PP_p^o \times \PP_p) \cap \Pairs(\PP_p)
 $, and 
 let
    \[
    \phi_{D,\ell,0}: \HH_\ell(\Delta_{C_1(D),C_2(D),E(D)}^0)\rightarrow \HH_\ell(\Delta_{C_1(D)}),
    \]
    and
    \[
\phi_{D,\ell,1}: \HH_\ell(\Delta_{C_1(D),C_2(D),E(D)})
\rightarrow \HH_\ell(\Delta_{C_2(D)}),
    \]
 denote the canonical isomorphisms induced by refinement, and 
    \[ \iota_{D,\ell}:\HH_\ell(\Delta_{C_1(D),C_2(D),E(D)}^0)\rightarrow \HH_\ell(\Delta_{C_1(D),C_2(D),E(D)})
    \]
    the homomorphism induced by inclusion.
}
\label{alg:main:triangulation:8}

\State{
    Denote 
    \[
    \widetilde{\phi}_{D,\ell}:=\phi_{D,\ell,1} \circ \iota_{D,\ell} \circ \phi_{D,\ell,0}^{-1}: \HH_\ell(\Delta_{C_1(D))}) \rightarrow \HH_\ell(\Delta_{C_2(D)}).
    \]
}
\label{alg:main:triangulation:9}

\State{
Compute $N_C = \dim \HH_\ell(\Delta_C)$, bases of the finite dimensional vector spaces
\[
\HH_\ell(\Delta_C), C \in \Cc(\mathcal{C}), 
\]
and matrices $\widetilde{M}_D$ with respect to the above bases for the linear maps 
\[
\widetilde{\phi}_{D,\ell},
\]
$D \in \Cc(\mathcal{D})$, such that $D \subset 
\Pairs(\R^p)
$.
}
\label{alg:main:triangulation:10}

\State{Let  $N = \max_{C \in \mathcal{C}} N_C$
and for each $D \in \Cc(\mathcal{D})$, such that $D \subset 
\Pairs(\R^p)$, set $M_D \in \kk^{N \times N}$ by padding $\widetilde{M}_D$
with $0$'s.
}
\label{alg:main:triangulation:11}
\State{Output:
\begin{enumerate}[(a)]
\item
$N$;
\item $\mathcal{C}$ and the tuple $(N_C)_{C \in \mathcal{C}}$;
\item 
$\mathcal{D}$ and the tuple $(M_D)_{D \in \Cc(\mathcal{D}), D \subset 
\Pairs(\R^p)}$.
\end{enumerate}
}
\label{alg:main:triangulation:12}
\COMPLEXITY
{
The complexity of the algorithm is bounded by 
\[
(s d)^{2^{O(n)}p^{O(1)}}.
\]
Moreover, the cardinalities and the degrees of the polynomials in them, of $\mathcal{C},\mathcal{D}$, are all bounded by
\[
(s d)^{2^{O(n)} p^{O(1)}}.
\]
}
\end{algorithmic}
\end{algorithm}

\begin{proposition}
    The output of Algorithm~\ref{alg:main:triangulation} satisfy Property~\ref{property:input-output}. 
    The complexity of Algorithm~\ref{alg:main:triangulation} is bounded by $(sd)^{2^{O(n)}p^{O(1)}}$,
    where $s = \card(\mathcal{P}) + \card(\mathcal{Q})$, and 
    $d = \max(\deg(\mathcal{P}),\deg(\mathcal{Q}))$.
\end{proposition}

\begin{proof}
    The required property of the output follow from the commutativity of the diagram ~\eqref{eqn:commutative}, which in turn follows from the definitions of the homomorphisms $\widetilde{\psi}_{\y,\ell}$ and 
    $\widetilde{\phi}_{C,C',D,\ell}$ and the correctness of the parametrized version of Algorithm~\ref{alg:triangulation}.

    The complexity bound follows from the complexity of Algorithm~\ref{alg:triangulation} given in \eqref{eqn:complexity:triangulation}.
    \end{proof}

\subsection{Proof of Theorem~\ref{thm:main}
}
As noted above the bound on the complexity of Algorithm~\ref{alg:main:triangulation} is doubly exponential in $n$ 
(and independent of $\ell$). We will now describe a more refined version
of Algorithm~\ref{alg:main:triangulation} whose complexity for any fixed $\ell$ is only singly exponential in $n$.

\medskip
The main new ingredient is to use a different algorithm -- namely, a parametrized version of the \emph{simplicial replacement algorithm} which is described in \cite{Basu-Karisani-1} instead of the parametrized version of the semi-algebraic triangulation algorithm. The main advantage of using the simplicial replacement algorithm is that its complexity for any fixed $\ell$ is only singly exponential in $n$ (unlike the semi-algebraic triangulation algorithm).

\medskip
In order to apply the simplicial replacement algorithm we first rely on another algorithm -- namely 
an algorithm for computing a cover by contractible semi-algebraic sets described in  \cite{Basu-Karisani-1} (Algorithm 1).

\subsubsection{Parametrized algorithm for computing cover by contractible sets}
To a first approximation this algorithm takes as input the description of a closed and bounded semi-algebraic set and produces as output 
descriptions of closed semi-algebraic sets, whose union is the given set and each of whom is semi-algebraically contractible. 
However, for technical reasons (explained in detail in \cite{Basu-Karisani-1}) the algorithm only succeeds in producing a cover by
semi-algebraically contractible sets of an infinitesimally larger
set than the given one -- which nonetheless has the same homotopy type.

\medskip
In order to describe more precisely the output of this algorithm because of this 
complication we need a technical detour. 

\medskip
\paragraph{\emph{Real closed extensions and Puiseux series}.}
We will need some
properties of Puiseux series with coefficients in a real closed field. We
refer the reader to \cite{BPRbook2} for further details.

\begin{notation}
  For $\R$ a real closed field we denote by $\R \left\langle \eps
  \right\rangle$ the real closed field of algebraic Puiseux series in $\eps$
  with coefficients in $\R$. We use the notation $\R \left\langle \eps_{1},
  \ldots, \eps_{m} \right\rangle$ to denote the real closed field $\R
  \left\langle \eps_{1} \right\rangle \left\langle \eps_{2} \right\rangle
  \cdots \left\langle \eps_{m} \right\rangle$. Note that in the unique
  ordering of the field $\R \left\langle \eps_{1}, \ldots, \eps_{m}
  \right\rangle$, $0< \eps_{m} \ll \eps_{m-1} \ll \cdots \ll \eps_{1} \ll 1$.
\end{notation}

\begin{notation}
\label{not:lim}
  For elements $x \in \R \left\langle \eps \right\rangle$ which are bounded
  over $\R$ we denote by $\lim_{\eps}  x$ to be the image in $\R$ under the
  usual map that sets $\eps$ to $0$ in the Puiseux series $x$.
\end{notation}

\begin{notation}
\label{not:extension}
  If $\R'$ is a real closed extension of a real closed field $\R$, and $S
  \subset \R^{k}$ is a semi-algebraic set defined by a first-order formula
  with coefficients in $\R$, then we will denote by $\E(S, \R') \subset \R'^{k}$ the semi-algebraic subset of $\R'^{k}$ defined by
  the same formula.
 It is well known that $\E(S, \R')$ does
  not depend on the choice of the formula defining $S$ 
  \cite[Proposition 2.87]{BPRbook2}.
\end{notation}

\begin{notation}
\label{not:monotone}
Suppose $\R$ is a real closed field,
and let $X \subset \R^k$ be a closed and bounded  semi-algebraic subset, and $X^+ \subset \R\la\eps\ra^k$
be a semi-algebraic subset bounded over $\R$. 
Let for $t \in \R, t >0$, $\widetilde{X}^+_{t} \subset \R^k$ denote the semi-algebraic
subset obtained by replacing $\eps$ in the formula defining $X^+$ by $t$, and it is
clear that  for $0 < t \ll 1$, $\widetilde{X}^+_t$ does not depend on the formula chosen. We say that $X^+$ is \emph{monotonically decreasing to $X$}, and denote $X^+ \searrow X$ if the following conditions are satisfied.
\begin{enumerate}[(a)]
\item
for all $0 < t < t'  \ll 1$,  $\widetilde{X}^+_{t} \subset  \widetilde{X}^+_{t'}$;
\item
\[
\bigcap_{t > 0} \widetilde{X}^+_{t} = X;
\]
or equivalently $\lim_\eps X^+ =  X$.
\end{enumerate}
More generally,
if $X \subset \R^k$ be a closed and bounded  semi-algebraic subset, and $X^+ \subset \R\la\eps_1,\ldots,\eps_m\ra^k$
a semi-algebraic subset bounded over $\R$,
we will say $X^+ \searrow X$ if and only if 
\[
X^+_{m+1} = X^+ \searrow X^+_m,  \;X^+_m \searrow X^+_{m-1}, \ldots, X^+_{2} \searrow X^+_1 = X,
\]
where for $i=1,\ldots, m$, $X^+_i = \lim_{\eps_i} X^+_{i+1}$. 
\end{notation}

The unparametrized version of the algorithm for computing a cover by contractible sets, takes as input a closed formula $\phi$, such that $\RR(\phi)$ is bounded, and
produces as output a tuple $\Phi = (\phi_1,\ldots,\phi_M)$ of closed formulas with coefficients in $\D[\bar\eps]$, such that
$\RR(\phi_j) \subset \R\la\bar\eps\ra^n, j \in [1,M]$ are semi-algebraically
contractible, and $\RR(\phi) = \searrow \bigcup_j \RR(\phi_j)$.

\medskip
We list the input and output and complexity of the parametrized version.
The complexity 
follows from an analysis of the complexity of the unparametrized version \cite[Algorithm 16.14]{BPRbook2}, and also a careful
analysis of \cite[Algorithm 16.5 (Parametrized bounded connecting)]{BPRbook2} that it  relies on, which in turn relies on the 
$\BElim_X$ operator whose properties were discussed earlier (Example~\ref{eg:BElim}). 

\begin{algorithm}[H]
\caption{(Parametrized algorithm for computing cover by contractible sets  )}
\label{alg:cover}
\begin{algorithmic}[1]
\INPUT
\Statex{
\begin{enumerate}[(a)]
\item
An element $R \in \D, R>0$;
\item 
a finite set $\mathcal{P} \subset \D[X_1,\ldots,X_n,Y_1,\ldots,Y_p]_{\leq d}$;
\item
a $\mathcal{P}$-closed formula $\phi(X_1,\ldots,X_n,Y_1,\ldots,Y_p)$, such that for each 
$\y \in \R^p$, $\RR(\phi(X,\y)) \subset \overline{B_n(0,R)}$.
\end{enumerate}
}
\OUTPUT
\Statex{
\begin{enumerate}[(a)]
\item a finite subset $\mathcal{F} \subset \D[Y_1,\ldots,Y_p]$;
    \item for each realizable sign condition $\sigma \in \{0,1,-1\}^{\mathcal{F}}$, 
    a tuple $\Phi_\sigma = (\phi_{\sigma,1},\ldots,\phi_{\sigma, M_\sigma})$, 
and where each $\phi_{\sigma,j} = \phi_{\sigma,j}(X_1,\ldots,X_n,Y_1,\ldots,Y_p)$ is a closed formulas with free variables $X_1,\ldots,X_n,Y_1,\ldots,Y_p$, 
with coefficients in $\D[\bar\eps]$, such that
for each $\y \in \RR(\sigma)$,
    and $j \in J$, $\RR(\phi_{\sigma,j}(X_1,\ldots,X_n,\y))$ is semi-algebraically contractible, 
    and $\RR(\phi_\sigma) = \searrow \bigcup_j \RR(\phi_{\sigma,j})$.
\end{enumerate}
}
\COMPLEXITY
The complexity of the algorithm is bounded by $(s d)^{ (p n)^{O(1)} } $,
where $s = \card(\mathcal{P})$.
\end{algorithmic}
\end{algorithm}

\subsubsection{Parametrized simplicial replacement algorithm}
In order to describe the input and output of the simplicial replacement algorithm we need a few preliminary definitions.
In the following we will restrict ourselves to the category of 
closed and bounded semi-algebraic sets and semi-algebraic continuous maps between them.

\begin{definition}[Semi-algebraic homological $\ell$-equivalences]
\label{def:equivalence-spaces}
We say that a semi-algebraic continuous map $f:X \rightarrow Y$ between two closed and bounded semi-algebraic sets is a semi-algebraic homological $\ell$-equivalence,
if the induced homomorphisms between the homology groups $f_*:\HH_i(X) \rightarrow \HH_i(Y)$ are isomorphisms for $0 \leq i \leq \ell$. 
\end{definition}

The relation of semi-algebraic homological $\ell$-equivalence as defined above is not an equivalence relation since it is not 
symmetric. In order to make it symmetric one needs to ``formally invert'' semi-algebraic homological $\ell$-equivalences.

\begin{definition}[Semi-algebraically homologically $\ell$-equivalent]
\label{def:ell-equivalent}
We will say that  \emph{$X$ is semi-algebraically homologically $\ell$-equivalent to $Y$}  (denoted $X \sim_\ell Y$), if and only if there exists 
closed and bounded semi-algebraic sets, $X=X_0,X_1,\ldots,X_n=Y$ and semi-algebraic homological $\ell$-equivalences  $f_1,\ldots,f_{n}$ as shown below:
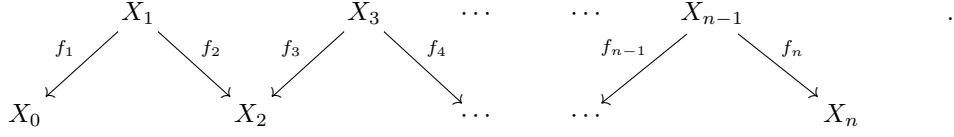
\[
\xymatrix{
&X_1 \ar[ld]_{f_1}\ar[rd]^{f_2} &&X_3\ar[ld]_{f_3} \ar[rd]^{f_4}& \cdots&\cdots&X_{n-1}\ar[ld]_{f_{n-1}}\ar[rd]^{f_{n}} & \\
X_0 &&X_2  && \cdots&\cdots &&  X_n&
}.
\]
It is clear that $\sim_\ell$ is an equivalence relation.
\end{definition}

We now extend Definition~\ref{def:equivalence-spaces} to semi-algebraic continuous maps 
between closed and bounded semi-algebraic  sets.

\begin{definition}[Homological $\ell$-equivalence between semi-algebraic maps]
\label{def:equivalence-diagrams}
Let $f_1:X_1 \rightarrow Y_1, f_2:X_2 \rightarrow Y_2$ be continuous semi-algebraic maps between closed and bounded semi-allgebraic sets. A semi-algebraic homological $\ell$-equivalence 
from $f_1$ to $f_2$ is then a pair $\phi = (\phi^{(1)},\phi^{(2)})$ where $\phi^{(1)}:X_1 \rightarrow X_2, \phi^{(2)}:Y_1 \rightarrow Y_2$ are semi-algebraic homological $\ell$-equivalences, and such that
$f_2 \circ \phi_1 = \phi_2 \circ f_1$.

\medskip
We will say that  a semi-algebraic map $f$ is \emph{semi-algebraically homologically $\ell$-equivalent} to a semi-algebraic map $g$ (denoted as before by $f \sim_\ell g$), if and only if there exists semi-algebraic continuous maps
 $f=f_0,f_1,\ldots,f_n=g$ between closed and bounded semi-algebraic sets, and semi-algebraic homological $\ell$-equivalences  $\phi_1,\ldots,\phi_{n}$ as shown below:
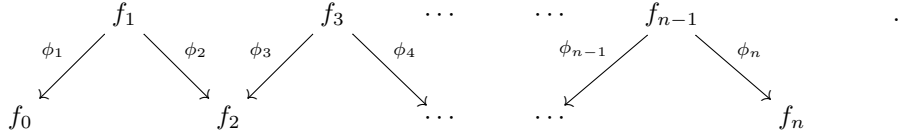
\[
\xymatrix{
&f_1 \ar[ld]_{\phi_1}\ar[rd]^{\phi_2} &&f_3\ar[ld]_{\phi_3} \ar[rd]^{\phi_4}& \cdots&\cdots&f_{n-1}\ar[ld]_{\phi_{n-1}}\ar[rd]^{\phi_{n}} & \\
f_0 &&f_2  && \cdots&\cdots &&  f_n&
}.
\]
It is clear that $\sim_\ell$ is an equivalence relation.
\end{definition}

\begin{definition}[Diagrams of closed and bounded semi-algebraic sets]
\label{def:diagram-of-spaces}
A diagram of closed and bounded semi-algebraic sets is a functor, $X:J \rightarrow \SAcat$, from a small category $J$ to the category of closed and bounded semi-algebraic sets and continuous semi-algebraic maps between them.
\end{definition}

We extend Definition~\ref{def:equivalence-spaces} to diagrams of closed and bounded semi-algebraic sets.
We denote by $\SAcat_\R$ the category of closed and bounded semi-algebraic subsets of $\R^n, n >0$  and continuous semi-algebraic maps between them.

\begin{definition}[Semi-algebraic homological $\ell$-equivalence between diagrams of closed and bounded semi-algebraic sets]
\label{equivalence-diagrams}
Let $J$ be a small category, and $X,Y: J \rightarrow \SAcat_\R$ be two functors. We say a natural transformation $f:X \rightarrow Y$ is an semi-algebraic homological $\ell$-equivalence, if
the induced maps, 
\[
f(j)_*: \HH_i(X(j)) \rightarrow \HH_i(Y(j))
\]
are isomorphisms for all $j \in J$ and $0 \leq i \leq \ell$.

We will say that  \emph{a diagram $X:J \rightarrow \SAcat_\R$ is $\ell$-equivalent to the diagram $Y:J \rightarrow \SAcat_\R$}  (denoted as before by $X \sim_\ell Y$), if and only if there exists diagrams
 $X=X_0,X_1,\ldots,X_n=Y:J \rightarrow \SAcat_\R$ and semi-algebraic homological $\ell$-equivalences  $f_1,\ldots,f_{n}$ as shown below:
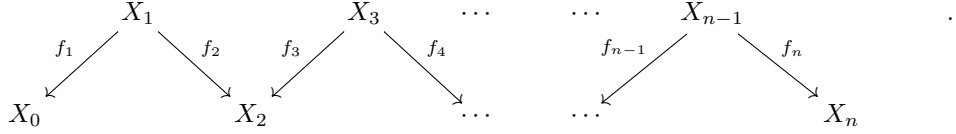
\[
\xymatrix{
&X_1 \ar[ld]_{f_1}\ar[rd]^{f_2} &&X_3\ar[ld]_{f_3} \ar[rd]^{f_4}& \cdots&\cdots&X_{n-1}\ar[ld]_{f_{n-1}}\ar[rd]^{f_{n}} & \\
X_0 &&X_2  && \cdots&\cdots &&  X_n&
}.
\]
It is clear that $\sim_\ell$ is an equivalence relation.
\end{definition}

One particular diagram will be important in what follows.

\begin{notation} [Diagram of various unions of a finite number of subspaces]
\label{not:diagram-Delta}
Let $J$ be a finite set, $A$ a closed and bounded semi-algebraic set, 
and $\mathcal{A} = (A_j)_{j \in J}$ a tuple of closed and bounded semi-algebraic subsets of $A$  indexed by $J$.

For any subset 
$J' \subset J$,
we denote 
\begin{eqnarray*}
\mathcal{A}^{J'} &=& \bigcup_{j' \in J'} A_{j'}, \\
\mathcal{A}_{J'} &=& \bigcap_{j' \in J'} A_{j'}, \\
\end{eqnarray*}

We consider $2^J$ as a category whose objects are elements of $2^J$, and whose only morphisms 
are given by: 
\begin{eqnarray*}
2^J(J',J'') &=& \emptyset  \mbox{ if  } J' \not\subset J'', \\
2^J(J',J'') &=& \{\iota_{J',J''}\} \mbox{  if } J' \subset J''.
\end{eqnarray*} 
We denote by $\Simp^J(\mathcal{A}):2^J \rightarrow \SAcat_\R$ the functor (or the diagram) defined by
\[
\Simp^J(\mathcal{A})(J') = \mathcal{A}^{J'}, J' \in 2^J,
\]
and
$\Simp^J(\mathcal{A})(\iota_{J',J''})$ is the inclusion map $\mathcal{A}^{J'} \hookrightarrow \mathcal{A}^{J''}$.
\end{notation}

\subsubsection{Parametrized algorithm for simplicial replacement}
The original  (i.e. unparametrized) algorithm takes as input a tuple
of closed formulas, $\Phi = (\phi_1,\ldots,\phi_M)$, such that the
realizations, $\RR(\phi_j) \subset \R^n, j \in J = [0,M]$ are all semi-algebraically contractible, and $m \geq 0$. It produces as output 
a simplicial complex $\Delta = \Delta^J$, having a subcomplex $\Delta^{J'}$ for each $J' \subset J$, with $\Delta^{J'} \subset \Delta^{J''}$ whenever
$J' \subset J'' \subset J$, and such that the diagram of inclusions
\[
(|\Delta^{J'}| \hookrightarrow |\Delta^{J''}|)_{J' \subset J'' \subset J}
\]
is homologically $m$-equivalent to the diagram of inclusions
\[
(\RR(\Phi^{J'}) \hookrightarrow \RR(\Phi^{J''}))_{J' \subset J'' \subset J}
\]
where for $J ' \subset J$, $\Phi^{J'} = \bigvee_{j \in J'} \phi_j$.

\medskip
We will need a \emph{parametrized version} of the above algorithm
which we will call the parametrized algorithm for simplicial replacement.
We describe the input, output and the complexity of this algorithm
below.

\medskip
The complexity of the parametrized version follows from analysing the 
complexity analysis of the unparametrized version in \cite{Basu-Karisani-1},
which has a recursive structure of depth $O(\ell)$. At each level of the recursion, there are calls to (the unparametrized version of) Algorithm~\ref{alg:cover}, 
on certain intersections of sets computed in the previous steps of the 
recursion. The complexity bound of the parametrized version now follows using the complexity bound, $(sd)^{( p n)^{O(1)}}$,
of Algorithm~\ref{alg:cover} (parametrized algorithm for computing cover by contractible sets), 
noting that the depth of the recursion in Algorithm~\ref{alg:poset}
is $O(\ell)$,
instead of using that of the  
the unparametrized version Algorithm~\ref{alg:cover} as is done in \cite{Basu-Karisani-1}. We omit the details since they are quite tedious.

\begin{algorithm}[H]
\caption{(Parametrized algorithm for simplicial replacement)}
\label{alg:poset}
\begin{algorithmic}[1]
\INPUT
\Statex{
\begin{enumerate}[(a)]
\item 
$\ell \geq 0$;
\item
An element $R \in \D, R>0$;
\item 
a finite set $\mathcal{P} \subset \D[X_1,\ldots,X_n,Y_1,\ldots,Y_p]_{\leq d}$;
\item
For each $j, 1 \leq j \leq M$, a 
$\mathcal{P}$-formula $\phi_j$, 
such that $\RR(\phi_j, \overline{B_k(0,R)})$ is 
closed and bounded, and 
semi-algebraically contractible.
\end{enumerate}
}
\OUTPUT
\Statex{
\begin{enumerate}[(a)]
\item a finite subset $\mathcal{F} \subset \D[Y_1,\ldots,Y_p]$;
\item for each $C \in \Cc(\mathcal{F})$, 
A simplicial complex $\Delta_C$, and subcomplexes $\Delta_{C, J'}, J' \subset J = [1,M]$, such that 
for each $J' \subset J'' \subset J$, $\Delta_{\sigma,J'} \subset \Delta_{C,J''}$
and the diagram of inclusions
$|\Delta_{C,J'}| \hookrightarrow |\Delta_{C,J''}|, J' \subset J'' \subset J$
is semi-algebraically homologically $\ell$-equivalent to
$\Simp^J(\RR(\Phi))$,
where $\Phi$ is defined by $\Phi(j) = \phi_j, j \in [1,M]$.
\end{enumerate}
}
\COMPLEXITY
The complexity of the algorithm is bounded by $(s d)^{n^{O(\ell)}p^{O(1)}} $,
where $s = \card(\mathcal{P})$.

\end{algorithmic}
\end{algorithm}

We now return to the proof of Theorem~\ref{thm:main}. 
We will prove the theorem by describing an algorithm with input and output as specified by the theorem and then proving its correctness and
upper bounds on the complexity. 
The following  algorithm will avoid using the semi-algebraic triangulation algorithm (and its inherently doubly exponential complexity). Instead,
we will use the parametrized algorithm for simplicial replacement 
(Algorithm~\ref{alg:poset}) described above.

\begin{algorithm}[H]
\caption{(Semi-algebraic multi-persistence using simplicial replacement)}
\label{alg:main:ss}
\begin{algorithmic}[1]
\INPUT
\Statex{
As stated in 
Section~\ref{subsec:input-output}.
}
\OUTPUT
\Statex{
As stated in 
Section~\ref{subsec:input-output}.
}

\PROCEDURE

\State{
Denote
\begin{eqnarray*}
\widetilde{S} &=& \{(\x,\y) \in \R^n \times \R^p \mid \x \in S, \mathbf{f}(\x) \preceq \y\}.    
\end{eqnarray*}
}

\State{
Let $\widetilde{S} = \RR(\widetilde{\phi})$, where $\widetilde{\phi}$
is a closed formula which is obtained from the closed formulas
describing $S$ and $\mathrm{graph}(\mathbf{f})$ in a straightforward way.
}

\State{
Treating $Y_1,\ldots,Y_p$ as parameters, use the 
parametrized algorithm for computing contractible covers using the
formula $\widetilde{\phi}$ as input to obtain the finite  set $\mathcal{F} \subset \D[Y_1,\ldots,Y_p]$ and the tuples $\Phi_C, 
C \in \Cc(\mathcal{F})$.
}

\algstore{myalg}
\end{algorithmic}
\end{algorithm}
 
\begin{algorithm}[H]
\begin{algorithmic}[1]
\algrestore{myalg}

\State{
Now use $(\mathcal{F}, (\Phi_C)_{C \in \Cc(\mathcal{F}}))$ as input to 
Algorithm~\ref{alg:poset} (parametrized algorithm for simplicial replacement)
to obtain a finite set $\mathcal{G} \subset \D[Y_1,\ldots,Y_p]$ and 
for each 
$C \in \Cc(\mathcal{F} \cup \mathcal{G})$
     a simplicial complex 
    $\Delta_C$ with property described earlier. 
    Let $\mathcal{C} = \mathcal{F} \cup \mathcal{G}$.
}

\State{
For each $C,C' \in \Cc(\mathcal{C})$, we obtain a semi-algebraic partition of $C \times C'  \cap 
(\PP_p^o \times \PP_p) \cap \Pairs(\PP_p)
$ as follows. Call Algorithm~\ref{alg:poset}(parametrized algorithm for simplicial replacement)
with input $\mathcal{C}(Y) \cup \mathcal{C}(Y'), (\Phi_C(X,Y) \mid \Phi_{C'}(X,Y')))$
with parameters $Y,Y'$, and where $\mid$ denotes concatenation. Let $\mathcal{D}_{C,C'} \subset \D[Y,Y']$
denote the finite set of polynomials in the output,
and for each $E \in \Cc(\mathcal{D}_{C,C'})$,
a simplicial complex $\Delta_{C,C',E}$,
and observe that $\Delta_C,\Delta_{C'}$ are both subcomplexes of the simplicial complex $\Delta'_{C,C',E}$. Moreover, there is a subcomplex $\Delta^0_{C,C',E}$ containing
$\Delta_C$ as a subcomplex, such that the inclusion
$|\Delta_C| \rightarrow |\Delta^0_{C,C',E}|$
is a homological $\ell$-equivalence.
}

\State{
Denote 
    \[
    \mathcal{D} = \mathcal{C}(Y) \cup \mathcal{C}(Y') \cup \bigcup_{(C,C')} \mathcal{D}_{C,C'},
    \] 
    where the last union is taken over
    all $(C,C') \in \Cc(\mathcal{C}) \times \Cc(\mathcal{C}')$ such that $C \times C' \cap 
    \Pairs(\R^p) 
    \neq \emptyset $.

    Note that for each $D \in \Cc(\mathcal{D})$, there
    exists uniquely defined $C_1(D) \in \Cc(\mathcal{C}),
    C_2(D) \in \Cc(\mathcal{C})$, and
    $E(D) \in \Cc(\mathcal{D}_{C,C'})$
    such that $D \subset C_1(D)) \times C_2(D)) \cap E(D)$.
 }
 
 \State{
 For each $D \in \Cc(\mathcal{D})$, such that $D \subset 
 (\PP_p^o \times \PP_p) \cap \Pairs(\PP_p)
 $, and 
 let
    \[
    \phi_{D,\ell,0}: \HH_i(\Delta_{C_1(D),C_2(D),E(D)}^0)\rightarrow \HH_\ell(\Delta_{C_1(D)}),
    \]
    and
    \[
\phi_{D,\ell,1}: \HH_\ell(\Delta_{C_1(D),C_2(D),E(D)})
\rightarrow \HH_\ell(\Delta_{C_2(D)}),
    \]
 denote the canonical isomorphisms induced by refinement, and 
    \[ \iota_{D,\ell}:\HH_\ell(\Delta_{C_1(D),C_2(D),E(D)}^0)\rightarrow \HH_\ell(\Delta_{C_1(D),C_2(D),E(D)})
    \]
    the homomorphism induced by inclusion.
}

\State{
    Denote 
    \[
    \widetilde{\phi}_{D,\ell}:=\phi_{D,\ell,1} \circ \iota_{D,\ell} \circ \phi_{D,\ell,0}^{-1}: \HH_\ell(\Delta_{C_1(D))}) \rightarrow \HH_\ell(\Delta_{C_2(D)}).
    \]
}
\label{alg:main:ss:4}

\State{
Compute $N_C = \dim \HH_\ell(\Delta_C)$, bases of the finite dimensional vector spaces
\[
\HH_\ell(\Delta_C), C \in \Cc(\mathcal{C}), 
\]
and matrices $\widetilde{M}_D$ with respect to the above bases for the linear maps 
\[
\widetilde{\phi}_{D,\ell},
\]
$D \in \Cc(\mathcal{D})$, such that $D \subset 
\Pairs(\R^p)
$.
}
\State{Output $N = \max_{C \in \mathcal{C}} N_C$
and for each $D \in \Cc(\mathcal{D})$, such that $D \subset 
\Pairs(\R^p)$, set $M_D \in \kk^{N \times N}$ by padding $\widetilde{M}_D$
with $0$'s.
}

\State{Output:
\begin{enumerate}[(a)]
\item
$N$;
\item $\mathcal{C}$ and the tuple $(N_C)_{C \in \mathcal{C}}$;
\item 
$\mathcal{D}$ and the tuple $(M_D)_{D \in \Cc(\mathcal{D}), D \subset 
\Pairs(\R^p)}$.
\end{enumerate}
}

\algstore{myalg}
\end{algorithmic}
\end{algorithm}
 
\begin{algorithm}[H]
\begin{algorithmic}[1]
\algrestore{myalg}

\COMPLEXITY{
The complexity of the algorithm is bounded by 
\[
(s d)^{n^{O(\ell)}p^{O(1)}}.
\]
Moreover, 
\[
\card(\mathcal{C}),\card(\mathcal{D}),\deg(\mathcal{C}),\deg(\mathcal{D})
\]
are all bounded by
\[
(s d)^{n^{O(\ell)}p^{O(1)}}.
\]
}
\end{algorithmic}
\end{algorithm}

\begin{proof}[Proof of correctness and complexity analysis of Algorithm~\ref{alg:main:ss}]
The correctness of the algorithm follows from the correctness of 
Algorithms~\ref{alg:cover} and \ref{alg:poset},
and the same arguments as in the proof correctness of Algorithm~\ref{alg:main:triangulation}.

The complexity bound follows from the complexities of Algorithms~\ref{alg:cover} and \ref{alg:poset}.
\end{proof}

\begin{proof}[Proof of Theorem~\ref{thm:main}]
Follows from the proof of correctness and complexity analysis of Algorithm~\ref{alg:main:ss}.
\end{proof}

\subsection{Proofs of Theorems~\ref{thm:main:mu} and \ref{thm:alg:main:mu}}
As mentioned previously (see Remark~\ref{rem:finite}), the key to effective constructiblity of the function $\mu_\ell(S,\mathbf{f})$ is the fact
for each $(\mathbf s, \mathbf t) \in \PP_p^o \times \PP_p$,
In Definition~\ref{def:persistent}, 
only finitely many distinct subspaces occur in the definition of
the subspaces $M^{\mathbf s, \mathbf t}(\mathbf P_{S,\mathbf{f},\ell}\!\restriction_{\PP_p})$ and $N^{\mathbf s,\mathbf t}(\mathbf P_{S,\mathbf{f},\ell}\!\restriction_{\PP_p})$. 
With this in mind we make the following definition.

\begin{definition}
\label{def:alg:mu}
Let $\mathcal D\subset \R[Y_1,\ldots,Y_p,Y'_1,\ldots,Y'_p]$ be finite.
\begin{enumerate}[(a)]
\item 
For $(\mathbf s,\mathbf t)\in (\PP_p^{o}\times \PP_p^{o})\cap \Pairs(\PP_p)$ we define
\[
\mathcal M^{\mathbf s,\mathbf t}(\mathcal D)
=\Bigl\{\,D\in \Cc(\mathcal D)\ \Bigm|\ \exists\, \mathbf s'\in \PP_p \text{ with } \mathbf s'\prec \mathbf s \text{ and } (\mathbf s',\mathbf t)\in D \Bigr\},
\]
and
$$
\displaylines{
\mathcal N^{\mathbf s,\mathbf t}(\mathcal D)
=
\Bigl\{\, (D_1,D_2)\in \Cc(\mathcal D)^2\ \Bigm|\ 
\exists\, \mathbf s',\mathbf t'\in \PP_p \text{ with } \mathbf s'\prec \mathbf s\preceq \mathbf t'\prec \mathbf t,\cr
(\mathbf s',\mathbf s)\in D_1,\ (\mathbf s,\mathbf t')\in D_2
\Bigr\}.
}
$$

\item
For $(\mathbf s,\mathbf t)\in (\PP_p^{o}\times \PP_p^{\max})\cap \Pairs(\PP_p)$ we define
$$
\displaylines{
\mathcal M^{\mathbf s,\mathbf t}(\mathcal D)
=\Bigl\{\, D\in \Cc(\mathcal D)\ \Bigm|\ \exists\ \mathbf t'\in \PP_p^{o}\text{ with } \mathbf s \preceq \mathbf t',\cr
 (\mathbf s,\mathbf t')\in D
\Bigr\}.
}
$$
\end{enumerate}
\end{definition}

\begin{definition}
\label{property:alg:mu}
Let $\mathcal D\subset \R[Y_1,\ldots,Y_p,Y'_1,\ldots,Y'_p]$ be finite. 
We say that a finite set $\mathcal D'\subset \R[Y_1,\ldots,Y_p,Y'_1,\ldots,Y'_p]$ is \emph{compatible with $\mathcal D$} if, for every  $D'\in \Cc(\mathcal D')$, the sets $\mathcal M^{\mathbf s,\mathbf t}(\mathcal D)$ and $\mathcal N^{\mathbf s,\mathbf t}(\mathcal D)$ are constant as $(\mathbf s, \mathbf t)$ ranges over $D'$.
\end{definition}

As mentioned earlier (see Remark~\ref{rem:finite}), the key to the effective constructibility of the barcode function $\mu_\ell(S,\mathbf f)$ is that, although the definitions in Definition~\ref{def:persistent} involve sums indexed by potentially infinite sets, only finitely many \emph{distinct} subspaces actually arise when the poset module comes from a semi-algebraic filtration. Concretely, for each $(\mathbf s,\mathbf t)\in (\PP_p^o\times \PP_p)\cap \Pairs(\PP_p)$, the subspaces
$M^{\mathbf s,\mathbf t}\!\bigl(\mathbf P_{S,\mathbf f,\ell}\!\restriction_{\PP_p}\bigr)$ and
$N^{\mathbf s,\mathbf t}\!\bigl(\mathbf P_{S,\mathbf f,\ell}\!\restriction_{\PP_p}\bigr)$
are determined by finitely many 
parts of a semi-algebraic partition of the set $\Pairs(\PP_p)$ -- namely, any semi-algebraic partition subordinate to the poset module $\mathbf{P}_{S,\mathbf{f},\ell}\!\restriction_{\PP_p}$ which we have shown is 
semi-algebraically constructible (Theorem~\ref{thm:main}).
With this in mind, we introduce the following combinatorial data extracted from a semi-algebraic partition of $\Pairs(\PP_p)$.

\begin{definition}
\label{def:alg:mu}
Let $\mathcal D\subset \R[Y_1,\ldots,Y_p,Y'_1,\ldots,Y'_p]$ be finite and assume that $Y_i\pm 1, Y_i' \pm 1, Y_i - Y_i' \in \mathcal{D}$ for $1 \leq i \leq p$.

\begin{enumerate}[(a)]
\item For $(\mathbf s,\mathbf t)\in (\PP_p^{o}\times \PP_p^{o})\;\cap\; \Pairs(\PP_p)$, define
$$
\displaylines{
\mathcal M^{\mathbf s,\mathbf t}(\mathcal D)
=
\Bigl\{\,D\in \Cc(\mathcal D), D \subset (\PP_p^o \times \PP_p) \cap \Pairs(\PP_p) \ \Bigm|\ \cr
\exists\, \mathbf s'\in \PP_p \text{ with } \mathbf s'\prec \mathbf s \text{ and } (\mathbf s',\mathbf t)\in D \Bigr\},
}
$$
and
$$
\displaylines{
\mathcal N^{\mathbf s,\mathbf t}(\mathcal D)
=
\Bigl\{\, (D_1,D_2)\in \Cc(\mathcal D)^2,\ D_1,D_2 \subset (\PP_p^o \times \PP_p) \cap \Pairs(\PP_p) \Bigm|\ 
\cr
\exists\, \mathbf s',\mathbf t'\in \PP_p \text{ with } \mathbf s'\prec \mathbf s\preceq \mathbf t'\prec \mathbf t,\ 
(\mathbf s',\mathbf s)\in D_1,\ (\mathbf s,\mathbf t')\in D_2
\Bigr\}.
}
$$

\item For $(\mathbf s,\mathbf t)\in (\PP_p^{o}\times \PP_p^{\max})\cap \Pairs(\PP_p)$, define
$$
\displaylines{
\mathcal M^{\mathbf s,\mathbf t}(\mathcal D)
=
\Bigl\{\, D\in \Cc(\mathcal D),\ D \subset (\PP_p^o \times \PP_p)\; \cap \; \Pairs(\PP_p)\  \Bigm|\
\cr
\exists\, \mathbf t'\in \PP_p^{o} \text{ with } \mathbf s\preceq \mathbf t' \prec \mathbf t
\ \text{ and }\ (\mathbf s,\mathbf t')\in D
\Bigr\}.
}
$$
\end{enumerate}
\end{definition}

\noindent\emph{Interpretation.}
Informally, $\mathcal M^{\mathbf s,\mathbf t}(\mathcal D)$ records which elements of $\Cc(\mathcal{D})$ of the partition of $\Pairs(\R^p)$ are hit by pairs of the form $(\mathbf s',\mathbf s)$ (or, in case $\mathbf t\in \PP_p^{\max}$, by pairs of the form $(\mathbf s,\mathbf t')$) with $\mathbf s'\prec \mathbf s$, while $\mathcal N^{\mathbf s,\mathbf t}(\mathcal D)$ records which \emph{ordered pairs} of elements of $\Cc(\mathcal{D})$ 
are hit simultaneously by a ``sandwich'' configuration
\[
\mathbf s'\prec \mathbf s \preceq \mathbf t' \prec \mathbf t,
\qquad
(\mathbf s',\mathbf s)\in D_1,\ (\mathbf s,\mathbf t')\in D_2.
\]

Finally, we want to partition $(\PP_p^o \times \PP_p)\; \cap \; \Pairs(\PP_p)$ so that on each part of the partition, the sets $\mathcal M^{\mathbf s,\mathbf t}(\mathcal D), \mathcal N^{\mathbf s,\mathbf t}(\mathcal D)$ stay constant.

To this end we define first:

\begin{definition}
\label{property:alg:mu}
Let $\mathcal D\subset \R[Y_1,\ldots,Y_p,Y'_1,\ldots,Y'_p]$ be finite.
We say that a finite set $\mathcal D'\subset \R[Y_1,\ldots,Y_p,Y'_1,\ldots,Y'_p]$ is \emph{compatible with $\mathcal D$} if for every $D'\in \Cc(\mathcal D')$, the sets
\[
\mathcal M^{\mathbf s,\mathbf t}(\mathcal D)
\quad\text{and}\quad
\mathcal N^{\mathbf s,\mathbf t}(\mathcal D)
\]
are constant as $(\mathbf s,\mathbf t)$ ranges over $D'$.
\end{definition}

The following proposition will play a key role in the proof of Theorem~\ref{thm:alg:main:mu} later.

\begin{proposition}
\label{prop:alg:main:mu}
     There exists an algorithm that takes as input a finite set
     $\mathcal{D} \subset \D[Y_1,\ldots,Y_p,Y_1',\ldots,Y_p']$, 
     and produces as output a finite set 
     \[
     \mathcal{D}' \subset \D[Y_1,\ldots,Y_p,Y_1'\ldots,Y_p']
     \]
     which is compatible with $\mathcal{D}$.
     Moreover, the complexity of the algorithm and the number and degrees of the polynomials are bounded by $(s d)^{p^{O(1)}}$,
     where $s = \card(\mathcal{D})$ and $d = \max_{P \in \mathcal{D}} \deg(P)$.
\end{proposition}

\begin{proof}
For each $D \in \Cc(\mathcal{D})$, let $\mathcal{P}_0(D)$ be a finite set of polynomials such that $D$ is a $\mathcal{P}_0(D)$-semi-algebraic set.
    Using the method described  in Section~\ref{subsec:enumerating-CC} the set of polynomials $\mathcal{P}_0(D)$,  as well as the $\mathcal{P}_0(D)$-semi-algebraic formula describing $D$, can be computed with complexity bounded by $(s d)^{p^{O(1)}}$, and such that $\card(\mathcal{P}_0(D), \deg(\mathcal{P}_0(D) \leq (s d)^{p^{O(1)}}$.

For $D \in \Cc(\mathcal{D})$, we denote 
\[
\mathcal{M}(D) = \{(\mathbf s, \mathbf t) \in (\PP_p^o \times \PP_p) \;\cap \; \Pairs(\PP_p)\; \mid\; D \in \mathcal{M}^{\mathbf s, \mathbf t}(\mathcal{D}) \}.
\]
Similarly,  for $D_1,D_2 \in \Cc(\mathcal{D})$, we denote 
\[
\mathcal{N}(D_1,D_2) = \{(\mathbf s, \mathbf t) \in (\PP_p^o \times \PP_p) \;\cap \; \Pairs(\PP_p)\; \mid\; (D_1,D_2) \in \mathcal{N}^{\mathbf s, \mathbf t}(\mathcal{D}) \}.
\]

Clearly, $\mathcal{M}(D),\mathcal{N}(D_1,D_2)$ are semi-algebraic sets. Moreover, using an efficient quantifier elimination algorithm (for example \cite[Algorithm 14.5 (Quantifier Elimination)]{BPRbook2} there exists an algorithm for computing semi-algebraic descriptions of these sets,
by some $\mathcal{P}_1(D), \mathcal{Q}_1(D_1,D_2)$-formulas, where 
$\mathcal{P}_1(D), \mathcal{Q}_1(D_1,D_2)$ are finite sets of  polynomials
whose cardinalities and degrees  are also  bounded singly exponentially.
Let 
\[
\mathcal{D}' = \bigcup_{D \in \Cc(\mathcal{D})} \mathcal{P}_1(D) \cup \bigcup_{D_1,D_2 \in \Cc(\mathcal{D})} \mathcal{Q}_1(D_1,D_2).
\]

It follows from the definitions of the sets 
$
\mathcal M^{\mathbf s,\mathbf t}(\mathcal D),
\mathcal N^{\mathbf s,\mathbf t}(\mathcal D)
$,
and also the definition of $\mathcal{D}'$, that $\mathcal{D}'$ is compatible with $\mathcal{D}$. The estimates on the degrees and cardinalities follow from the 
complexity of the quantifier elimination algorithm (\cite[Algorithm 14.5 (Quantifier Elimination)]{BPRbook2}, and the fact that $\card(\Cc(\mathcal{D}))$ is also bounded by $(O(s d))^{2 p}$ (see Theorem~\ref{thm:OPTM} below).
\end{proof}

In the algorithm we describe next we will use the following notation.
\begin{notation}
    For $D' \in \Cc(\mathcal{D}')$ in Proposition~\ref{prop:alg:main:mu} we will denote by 
    $\mathcal{M}^{D'}$ (resp. $\mathcal{N}^{D'}$) the set
    $\mathcal{M}^{\mathbf s, \mathbf t}(\mathcal{D})$ (resp. $\mathcal{N}^{\mathbf s, \mathbf t}(\mathcal{D})$) for some (and so all) $(\mathbf s, \mathbf t)\in D'$.
\end{notation}
We will also use heavily Notation~\ref{not:matrix} (linear maps corresponding to matrices).
\begin{algorithm}[H]
\caption{(Computing barcodes of semi-algebraic multi-filtrations)}
\label{alg:main:ss:mu}
\begin{algorithmic}[1]
\INPUT
\Statex{
As stated in Theorem~\ref{thm:alg:main:mu}.
}
\OUTPUT
\Statex{
As stated in Theorem~\ref{thm:alg:main:mu}.
}

\PROCEDURE
\State{
Using Algorithm~\ref{alg:main:ss} with input $\mathcal{P},\mathcal{Q},\Phi,\Psi,\ell$
compute $N, \mathcal{C},\mathcal{D}$, and the tuples
$(N_C)_{C \in \Cc(\mathcal{C})}$ and $(M_D)_{D \in \Cc(\mathcal{D}), D \subset \Pairs(\R^p)}$.
}

\State{
Using the  Proposition~\ref{prop:alg:main:mu} compute a finite subset $\mathcal{D}' \subset \D[Y_1,\ldots,Y_p,Y_1',\ldots,Y_p']$ which is adapted to $\mathcal{D}$. 
}

\State{
For each $D' \in \Cc(\mathcal{D}')$, denote by $D_0(D') \in \Cc(\mathcal{D})$ the unique
element of $\Cc(\mathcal{D})$ such that $D' \subset D_0(D')$,
and $N_1(D') = N_{C_1(D')}, N_2(D') = N_{C_2(D')}$, where
$C_1(D'), C_2(D') \in \Cc(\mathcal{C})$ are such that
$D_0(D') \subset C_1(D') \times C_2(D')$.
}

\algstore{myalg}
\end{algorithmic}
\end{algorithm}
 
\begin{algorithm}[H]
\begin{algorithmic}[1]
\algrestore{myalg}

\State{
For each $D' \in \mathcal{D}'$ such that $D' \subset (\PP^o \times \PP_p^o)
\;\cap \; \Pairs(\PP_p)$,
compute using standard algorithms from linear algebra,
\begin{eqnarray*}
\mu_\ell(D') &=& \dim 
\left(
\sum_{D \in \mathcal{M}^{D'}} 
L_{M_{D'}}^{-1}
\left(
\mathrm{colspace}(M_D)
\right)
\right) 
\\
&& - 
\dim
\left(
\sum_{(D_1,D_2) \in \mathcal{N}^{D'}} 
L_{M_{D'}}^{-1}
\left(
\mathrm{colspace}(M_{D_2} \cdot M_{D_1})
\right)
\right).
\end{eqnarray*}
}

\State{
For each $D' \in \mathcal{D}'$ such that $ D' \subset (\PP^o \times \PP_p^{\max})
\;\cap \; \Pairs(\PP_p)$,
compute 
\begin{eqnarray*}
\mu_\ell(D') = 
N_1(D') - 
\dim
\left(
\sum_{(D_1,D_2)  \in \mathcal{M}^{D'}} 
L_{M_{D_1}}^{-1}
\left(
\mathrm{colspace}(M_{D_2}\cdot M_{D_1})
\right)
\right).
\end{eqnarray*}
}

\State{
Output $\mathcal{D}'$, and the tuple of pairs
\[
\left(D,\mu_\ell(D)\right)_{D \in \Cc(\mathcal{D}'), D \subset (\PP^o_p \times \PP_p) \cap \Pairs(\PP_p)}.
\]
}

\COMPLEXITY
{
The complexity of the algorithm is bounded by 
\[
(s d)^{n^{O(\ell)}p^{O(1)}}.
\]
Moreover, 
\[
\card(\mathcal{C}),\card(\mathcal{D}),\deg(\mathcal{C}),\deg(\mathcal{D})
\]
are all bounded by
\[
(s d)^{n^{O(\ell)}p^{O(1)}}.
\]
}
\end{algorithmic}
\end{algorithm}

\begin{proof}[Proof of correctness of Algorithm~\ref{alg:main:ss:mu}]
The correctness of Algorithm~\ref{alg:main:ss:mu} follows from the correctness of Algorithm~\ref{alg:main:ss}, Proposition~\ref{prop:alg:main:mu} and Definition~\ref{def:sa-mp-barcode-mu}.
\end{proof}
\begin{proof}[Proof of Theorem~\ref{thm:alg:main:mu}]
The theorem follows from the correctness and complexity analysis of 
Algorithm~\ref{alg:main:ss:mu}.
\end{proof}

\section{Proofs of Theorems~\ref{thm:speed} and \ref{thm:speed:uniform}}
\label{sec:proof:3-4}
In the proofs of Theorems~\ref{thm:speed} and \ref{thm:speed:uniform}
we will need the following basic result from real algebraic geometry giving an upper bound on the sum of the (zero-th) Betti numbers of the realizations
of all realizable sign conditions of a finite set of polynomials.

\begin{theorem} \cite[Theorem 7.30]{BPRbook2}
\label{thm:OPTM}
    Let $\mathcal{P} \subset \R[X_1,\ldots,X_n]_{\leq d}$ with $\card(\mathcal{P}) = s$. Then,
    \[
    \card(\Cc(\mathcal{P})) \leq \sum_{j=1}^{n} \binom{s}{j} 4^j d(2d-1)^{n-1}.
    \]
\end{theorem}

\begin{proof}[Proof of Theorem~\ref{thm:speed}]
    Using Theorem~\ref{thm:main} we have that the poset module
    $\mathbf{P}_{S,\mathbf{f},\ell}$ is semi-algebraically constructible
    (recall Definition~\ref{def:constructible-poset-module}).
    Thus, there exist $M > 0$, and a 
    semi-algebraically constructible function $F:\Pairs(\R^p) \rightarrow \mathbf{k}^{M \times M}$ 
    associated to 
    $\mathbf{P}_{S,\mathbf{f},\ell}$. Moreover, since by Theorem~\ref{thm:main}, the complexity of $\mathbf{P}_{S,\mathbf{f},\ell}$ is bounded by 
    $(s d)^{n^{O(\ell)}p^{O(1)}}$
    we can assume that there exists a finite set
    of polynomials,
    $\mathcal{D} \subset \R[Y_1,\ldots,Y_p,Y_1',\ldots,Y_p]$, 
    such that the partition $(D)_{D \in \Cc(\mathcal{D}), D \subset \Pairs(\R^p)}$ is subordinate to $F$, with
    \begin{equation}
    \label{eqn:C'}
    C'  := \max(\card(\mathcal{D}),\deg(\mathcal{D})) \leq
    (s d)^{n^{O(\ell)}p^{O(1)}}
    \end{equation}
    
    Now let 
    \[
    \widetilde{\mathcal{D}} \subset \R[Y_{1,1},\ldots,Y_{1,p},\ldots, Y_{N,1},\ldots,Y_{N,p}] 
    \]
    be defined as follows.
    \[
    \widetilde{\mathcal{D}} := \bigcup_{1 \leq i,j \leq N} \{P(Y_{i,1},\ldots,Y_{i,p}, Y_{j,1},\ldots,Y_{j,p}) \mid P \in \mathcal{D} \}.
    \]
    Then, 
    \[
    \card(\widetilde{\mathcal{D}}) \leq C' \cdot N^2, 
    \]
    \[
    \deg(\widetilde{\mathcal{D}}) \leq C',
    \]
    and the number of variables in the polynomials in 
    $\widetilde{\mathcal{D}}$ is $p N$.
    
    It follows from the fact that $F$ is associated to the poset module
    $\mathbf{P}_{S,\mathbf{f},\ell}$, and that the partition $(D)_{D \in \Cc(\mathcal{D})}$ is subordinate to $F$,
    that for each $T = (\y_1,\ldots,\y_N) \in (\R^p)^N$, the strong equivalence class of the finite poset module $\mathbf{P}_{S,\mathbf{f},T,\ell}$ is determined by the map
    \[
     \{(i,j) \mid \y_i \preceq \y_j\} 
     \rightarrow \Cc(\mathcal{D}), 
      (i,j) \mapsto D(\y_i,\y_j),
    \]
    where $D(\y_i,\y_j)$ is the unique element of $\Cc(\mathcal{D})$ containing $(\y_i,\y_j)$.
   
    Let $\widetilde{D}_T \in \Cc(\widetilde{\mathcal{D}})$ such that
    $(\y_1,\ldots,\y_N) \in \widetilde{D}_T$.
    Denote by $\pi_{i,j}: (\R^p)^N \rightarrow \R^p \times \R^p$ the projection map on to the $(i,j)$-th coordinate (tuples). 
    
    Then, for each $i,j$ with $\y_i \preceq \y_j$, $\pi_{ij}(\widetilde{D}(T)) = D(\y_i,\y_j)$.

    Thus, the strong equivalence class of $\mathbf{P}_{S,\mathbf{f},T,\ell}$ is determined by 
    $\widetilde{D}_T$, and hence the number of possibilities for the strong equivalence class of $\mathbf{P}_{S,\mathbf{f},T,\ell}$ is bounded by $\card(\Cc(\widetilde{\mathcal{D}}))$. 
    
    Using Theorem~\ref{thm:OPTM} and \eqref{eqn:C'} we obtain
    \begin{eqnarray*}
\card(\Cc(\widetilde{\mathcal{D}})) &\leq & 
\sum_{j=1}^{pN} \binom{C' N^2}{j} \cdot 4^j \cdot  C'(2C'-1)^{ p N-1 } \\
&\leq &
    \sum_{j=1}^{pN} \binom{C N^2}{j} \cdot C^{ p N}
    \end{eqnarray*}
with $C = 8 C' = (sd)^{(np)^{O(\ell)}}$ (using  \eqref{eqn:C'}).

In order to get the asymptotic upper bound observe that

\begin{eqnarray*}
 \sum_{j=1}^{pN} \binom{C N^2}{j} \cdot C^{ p N} &\leq&
 (p N) \cdot \binom{C N^2}{p N} \cdot C^{p N} \\
 &\leq &
 (p N)\cdot \left(\frac{e C N^2}{p N}\right)^{p N} \cdot  C^{p N } \\
 &\leq& (p N) \cdot  \left(\frac{e C^2}{p}\right)^{p N} \cdot N^{p N} \\
 &=&\left (N^{o(1) pN}\right) \cdot N^{p N} \\
 &=& N^{(1+o(1))p N},
 \end{eqnarray*}
where in the second step we have used the inequality
\[
\binom{m}{k} \leq \left(\frac{em}{k}\right)^k
\]
valid for all $m,k$ with $0 \leq k \leq m$.

\end{proof}

\begin{proof}[Proof of Theorem~\ref{thm:speed:uniform}]
For $s,d,n > 0$, let
\[
M(s,d,n) :=  s \times \binom{n + d}{d},
\]
denote the number of monomials in $s$ polynomials in $n$ variables of degree $d$.

Let $(A_i,A_i')_{1 \leq i \leq s'}, (B_i,B_i')_{1 \leq i \leq s''}$ denote Boolean variables, 
with $s' + s''= s$, and let 
\[
\widetilde{\Phi}(A_1,A_1'\ldots,A_{s'},A_{s'}'), \widetilde{\Psi}(B_1,B_1',\ldots,B_{s''},B_{s''}')
\]
denote two Boolean formulas.

Notice that there as most $2^{2^{2s}}$ many non-equivalent Boolean formulas in $2s$ Boolean indeterminates.

Now given, 
\[
\mathcal{P} = (P_1,\ldots,P_{s'}) \in (\R[X_1,\ldots,X_n])^{s'},
\]
and 
\[
\mathcal{Q} = (Q_1,\ldots,Q_{s''}) \in (\R[Y_1,\ldots,Y_p,Y_1',\ldots,Y_p'])^{s''},
\]
we will denote by $\widetilde{\Phi}(\mathcal{P})$ (resp.
$\widetilde{\Psi}(\mathcal{Q})$) the formulas obtained by substituting 
in $\widetilde{\Phi}$ (resp. $\widetilde{\Psi})$,
$A_i$ by $P_i \geq 0$, $A_i'$ by $P_i \leq 0$, (resp. $B_i$ by $Q_i \geq 0$, $B_i'$ by $Q_i \leq 0$).

For 
\[
\mathcal{P} \in (\R[X_1,\ldots,X_n]_{\leq d})^{s'}
\]
and 
\[
\mathcal{Q} \in (\R[Y_1,\ldots,Y_p,X_1,\ldots,X_n]_{\leq d})^{s''},
\]
we will identify $(\mathcal{P},\mathcal{Q})$ with the point in 
$\R^{M(s',d,n)}\times \R^{M(s'',d,n+p)}$ whose coordinates give the coefficients of the polynomials in $\mathcal{P},\mathcal{Q}$.
We will denote the vector of coefficients of $\mathcal{P}$ by
$A$ and those of $\mathcal{Q}$ by $B$, for $A = \mathbf{a}, B = \mathbf{b}$, we will denote by 
$\mathcal{P}_{\mathbf{a}},\mathcal{Q}_{\mathbf{b}}$ the corresponding tuples of polynomials
having these coefficients.

For a pair. $(\widetilde{\Phi},\widetilde{\Psi})$, of Boolean formulas, 
$(\mathcal{P},\mathcal{Q}) \in \R^{M(s',d,n)}\times \R^{M(s'',d,n+p)}$,
we denote by 
\[
S_{\widetilde{\Phi}}(\mathcal{P}) = \RR(\widetilde{\Phi}(\mathcal{P})),
\]
and 
by $\mathbf{f}_{\widetilde{\Psi}}(\mathcal{Q})$, the semi-algebraic map
$\R^n \rightarrow \R^p$, such that
\[
\mathrm{graph}(\mathbf{f}) = \RR(\widetilde{\Psi}(\mathcal{Q})).
\]

Treating the coefficient vectors $A,B$ of $\mathcal{P},\mathcal{Q}$ in
the input of Algorithm~\ref{alg:main:ss} as parameters, 
and using a parametrized version of  Algorithm~\ref{alg:main:ss} (see discussion in Subsection~\ref{subsec:parametrized}),
we obtain a finite set of polynomials $\mathcal{H} \subset \R[A,B]$, and for each $H \in \Cc(\mathcal{H})$,
finite sets $\mathcal{C}_H \subset \R[A,B,Y_1,\ldots,Y_p]$,
$\mathcal{D}_H \subset \R[A,B,Y_1,\ldots,Y_p,Y_1',\ldots,Y_p']$,
such that
\[
\mathcal{C}_H(\mathbf{a},\mathbf{b}) \subset \R[Y_1,\ldots,Y_p],
\mathcal{D}_H(\mathbf{a},\mathbf{b}) \subset \R[Y_1,\ldots,Y_p,Y_1',\ldots,Y_p'].
\]
belongs to
the output of Algorithm~\ref{alg:main:ss} with input 
\[
\mathcal{P}_{\mathbf{a}},\mathcal{Q}_{\mathbf{b}}, \widetilde{\Phi}(\mathcal{P}_{\mathbf{a}}),\widetilde{\Psi}(\mathcal{Q}_{\mathbf{b}}),
\ell
\]
for all $(\mathbf{a},\mathbf{b}) \in H$.

It now follows using the same argument and notation as in the proof of 
Theorem~\ref{thm:speed}, that the number of strong equivalence classes amongst the poset modules 
\[
\left(\mathbf{P}_{S_{\widetilde{\Phi}(\mathcal{P}_{\mathbf{a}})}, \mathbf{f}_{\widetilde{\Psi}(\mathcal{Q}_{\mathbf{b}})},T,\ell}\right)_{(\mathbf{a},\mathbf{b}) \in 
\R^{M(s',d,n)}\times \R^{M(s'',d,n+p)}, T \in (\R^p)^N
}
\]
is bounded by
\[
\card(\Cc(\mathcal{H} \cup \bigcup_{H \in \Cc(\mathcal{H})} \widetilde{\mathcal{D}_H})),
\]
where
\[
    \widetilde{\mathcal{D}_H} := \bigcup_{1 \leq i,j \leq N} \{P(Y_{i,1},\ldots,Y_{i,p}, Y_{j,1},\ldots,Y_{j,p}) \mid P \in \mathcal{D}_H \}.
    \]

Note that 
\begin{eqnarray*}
 C &:=& \card(\bigcup_{H \in \Cc(\mathcal{H})} \mathcal{D}_H),\\
 C' &:=&  \card(\mathcal{H}), \\
 K &:=& \deg(\mathcal{H} \cup \bigcup_{H \in \Cc(\mathcal{H})} \mathcal{D}_H),
\end{eqnarray*}
are all bounded in terms of $s,d,n,\ell,p$ but independent of $N$.

Also note that the number of variables in the polynomials
$\mathcal{H}$ equals $M$, and 
that in the polynomials in 
\[
\bigcup_{H \in \Cc(\mathcal{H})} \widetilde{\mathcal{D}_H}
\]
equals $p N + M$,
where $M =M(s',d,n) + M(s'',d,n+p)$.

Now using
using Theorem~\ref{thm:OPTM} we obtain that
\begin{eqnarray*}
\card(\Cc(\mathcal{H} \cup \bigcup_{H \in \Cc(\mathcal{H})} \widetilde{\mathcal{D}_H})) &\leq&
\sum_{j=1}^{p N + M} \binom{C N^2 + C'}{p N + M} 4^j K(2K-1)^{p N + M-1} \\
&=& N^{(1 + o(1))N}.
\end{eqnarray*}

Finally the theorem follows from the fact there at most 
$\sum_{s' + s'' =s } 2^{2^{2s'}} \times 2^{2^{2s''}}$ pairs
of Boolean formulas $(\widetilde{\Phi},\widetilde{\Psi})$ to consider.
\end{proof}

\bibliographystyle{amsplain}
\bibliography{master}
 
\end{document}